\documentclass[12pt,leqno]{amsart}
\usepackage{amssymb,amsmath,amsthm,bm}
\usepackage{graphicx}
\usepackage{color}
\usepackage[textsize=tiny]{todonotes}
\usepackage[normalem]{ulem}
\usepackage[bookmarksopen,bookmarksdepth=3,colorlinks,citecolor=red,pagebackref,hypertexnames=true]{hyperref}
\usepackage[msc-links,nobysame,non-sorted-cites, initials]{amsrefs}
\usepackage[inline]{enumitem}
\usepackage{mathtools}
\mathtoolsset{showonlyrefs}

\definecolor{darkblue}{RGB}{0,0,160}

\usepackage[lining]{libertine}   
\usepackage{cabin}
\usepackage[libertine]{newtxmath}

\usepackage[T1]{fontenc}

\def\e{{\rm e}}
\def\eps{\varepsilon}

\def\d{{\rm d}}

\def\R {\mathbb{R}}
\def\N {\mathbb{N}}

\def\I {{\mathcal I}}

\def\M {{\mathrm M}}

\def\1 {{\mbox{\boldmath 1}}}

\def \l {\langle}

\def \r {\rangle}

\def \and{\quad\text{and}\quad}

\def\ind{\cic{1}}
\newcommand{\cic}{\bm}
\def \no#1#2#3 {{\bf #1} (#3), #2.}
\def \eds#1#2#3 {#1, #2, #3.}
\newcounter{counter}

\numberwithin{equation}{section}
\numberwithin{counter2}{section}
\newtheorem{proposition}[subsection]{Proposition}
\newtheorem{theorem}[subsection]{Theorem}
\newtheorem{corollary}{Corollary}
\newtheorem{lemma}[subsection]{Lemma}

\theoremstyle{definition}
\newtheorem{definition}[subsection]{Definition}
\newtheorem*{remark*}{Remark}
\newtheorem*{warn*}{A word of warning}

\newtheorem{remark}[subsection]{Remark} 
\theoremstyle{plain}
\numberwithin{figure}{section}

\makeatletter
\let\save@mathaccent\mathaccent
\newcommand*\if@single[3]{%
  \setbox0\hbox{${\mathaccent"0362{#1}}^H$}%
  \setbox2\hbox{${\mathaccent"0362{\kern0pt#1}}^H$}%
  \ifdim\ht0=\ht2 #3\else #2\fi
  }
\newcommand*\rel@kern[1]{\kern#1\dimexpr\macc@kerna}
\newcommand*\widebar[1]{\@ifnextchar^{{\wide@bar{#1}{0}}}{\wide@bar{#1}{1}}}
\newcommand*\wide@bar[2]{\if@single{#1}{\wide@bar@{#1}{#2}{1}}{\wide@bar@{#1}{#2}{2}}}
\newcommand*\wide@bar@[3]{%
  \begingroup
  \def\mathaccent##1##2{%
    \let\mathaccent\save@mathaccent
    \if#32 \let\macc@nucleus\first@char \fi
    \setbox\z@\hbox{$\macc@style{\macc@nucleus}_{}$}%
    \setbox\tw@\hbox{$\macc@style{\macc@nucleus}{}_{}$}%
    \dimen@\wd\tw@
    \advance\dimen@-\wd\z@
    \divide\dimen@ 3
    \@tempdima\wd\tw@
    \advance\@tempdima-\scriptspace
    \divide\@tempdima 10
    \advance\dimen@-\@tempdima
    \ifdim\dimen@>\z@ \dimen@0pt\fi
    \rel@kern{0.6}\kern-\dimen@
    \if#31
      \overline{\rel@kern{-0.6}\kern\dimen@\macc@nucleus\rel@kern{0.4}\kern\dimen@}%
      \advance\dimen@0.4\dimexpr\macc@kerna
      \let\final@kern#2%
      \ifdim\dimen@<\z@ \let\final@kern1\fi
      \if\final@kern1 \kern-\dimen@\fi
    \else
      \overline{\rel@kern{-0.6}\kern\dimen@#1}%
    \fi
  }%
  \macc@depth\@ne
  \let\math@bgroup\@empty \let\math@egroup\macc@set@skewchar
  \mathsurround\z@ \frozen@everymath{\mathgroup\macc@group\relax}%
  \macc@set@skewchar\relax
  \let\mathaccentV\macc@nested@a
  \if#31
    \macc@nested@a\relax111{#1}%
  \else
    \def\gobble@till@marker##1\endmarker{}%
    \futurelet\first@char\gobble@till@marker#1\endmarker
    \ifcat\noexpand\first@char A\else
      \def\first@char{}%
    \fi
    \macc@nested@a\relax111{\first@char}%
  \fi
  \endgroup
}
\makeatother

\begin{document}

\title[maximal directional Hilbert transform]{On the maximal directional Hilbert transform \\ in three dimensions}

\author[F. Di Plinio]{Francesco Di Plinio} \address{\noindent Department of Mathematics, University of Virginia, Box 400137, Charlottesville, VA 22904, USA}
\email{\href{mailto:francesco.diplinio@virginia.edu}{\textnormal{francesco.diplinio@virginia.edu}}}
\thanks{F. Di Plinio was partially supported by the National Science Foundation under the grant   NSF-DMS-1650810.}

\author[I. Parissis]{Ioannis Parissis}
\address{Departamento de Matem\'aticas, Universidad del Pais Vasco, Aptdo. 644, 48080 Bilbao, Spain and Ikerbasque, Basque Foundation for Science, Bilbao, Spain}

\email{\href{mailto:ioannis.parissis@ehu.es}{\textnormal{ioannis.parissis@ehu.es}}}
\thanks{I. Parissis is supported by grant  MTM2014-53850 of the Ministerio de Econom\'ia y Competitividad (Spain), grant IT-641-13 of the Basque Government, and IKERBASQUE}


%%%%%%%%%%%%%%%%%%%%%%%%%%%%%% ABSTRACT ABSTRACT ABSTRACT
\begin{abstract} 
We establish the sharp growth rate, in terms of cardinality, of  the $L^p$ norms of the maximal Hilbert transform $H_\Omega$  along  finite subsets  of a finite order lacunary set of directions $\Omega \subset \R^3$, answering a question of Parcet and Rogers in dimension $n=3$. Our result is  the first sharp estimate for maximal directional singular integrals  in dimensions greater than 2. 

The proof relies on a representation of the maximal directional Hilbert transform  in terms of a model maximal operator associated to compositions of two-dimensional angular multipliers,   as well as on the usage of    weighted norm inequalities, and their extrapolation, in the directional setting. 
\end{abstract}
%%%%%%%%%%%%%%%%%%%%%%%%%%%%%% ABSTRACT ABSTRACT ABSTRACT
\maketitle

%%%%%%%%%%%%%%%%%%%%%%%%%%%%%% SECTION SECTION SECTION
\section{Introduction}
Let $n\geq 2$. The   Hilbert transform along a direction   ${\omega}\in S^{n-1}$ acts on Schwartz functions on $\R^n$ by the principal value integral
\[
H_{{\omega}}f(x) \coloneqq  \mathrm{p.v.} \int_\R f(x+t{\omega}) \frac{\d t}{t}, \qquad x \in \R^n.
\]
If $\Omega\subset S^{n-1}$, we may define the corresponding \emph{maximal directional Hilbert transform}
\begin{equation}
\label{eq:MHT}
H_\Omega f \coloneqq \sup_{{\omega}\in \Omega} |H_{{\omega}} f|.
\end{equation}
The main result of this paper is the following sharp estimate in the three-dimensional case.
\begin{theorem}\label{th:mainintro} Let $\Omega\subset S^{2}$ be a finite order lacunary set \cite{PR2}. Then for all $1<p<\infty$ 
\begin{equation}
\label{eq:th:mainintro}
\sup_{\substack{O\subset \Omega \\ \# O\leq N}} \left\|H_{O}\right\|_{L^p(\R^3)\to L^p(\R^3)} \leq C \sqrt{\log N}.
\end{equation}
The positive constant $C$ may depend on $1<p<\infty$ and on the lacunary order of $\Omega$ only.
\end{theorem}
We stress that the supremum in Theorem \eqref{eq:th:mainintro}  is taken over all subsets $O$   having finite cardinality $N$ of a given finite order lacunary set $\Omega$, which may be infinite. Theorem \ref{th:mainintro}   is in fact the Lebesgue measure case of a more general sharp weighted norm inequality which is a natural byproduct of our proof techniques, and is detailed in Corollary \ref{cor:main} for the interested reader. In \cite{LMP}  Laba, Marinelli, and Pramanik have extended to dimensions $n\geq 2$ the lower bound (due to Karagulyan \cite{Karag} in the case $n=2$)
\begin{equation}
\label{eq:Karaglb}
\inf_{ \# \Omega= N}\left\| H_{\Omega}\right\|_{L^2(\R^n)\to L^2(\R^n)} = c_{n} \sqrt{\log N},
\end{equation}
where the infimum is taken over \emph{all} sets $\Omega\subset S^{n-1}$ of finite cardinality $N$. A comparison with the upper bound of Theorem \ref{th:mainintro} and  interpolation reveals that the dependence on the cardinality of the set of directions in our theorem is sharp for all $1<p<\infty$. In fact, \cite{LMP} proves the analogue of \eqref{eq:Karaglb} for all $1<p<\infty$.

%%%%%%%%%%%%%%%%%%%%%%%%%%%%%% SECTION SECTION SECTION
\subsection{Maximal and singular integrals along sets of directions}The study of car\-di\-na\-li\-ty-free, or sharp bounds for the companion directional maximal operator to \eqref{eq:MHT}
\begin{equation}
\label{eq:MF}
\M_\Omega f(x) = \sup_{{\omega}\in \Omega} \M_{{\omega}} f(x), \qquad \M_{{\omega}}f(x) \coloneqq  \sup_{\eps>0} \frac{1}{2\eps} \int_{-\eps}^\eps |f(x+t{\omega})| \d t, \qquad x \in \R^n,
\end{equation}
is a classical subject in real and harmonic analysis, with deep connections to multiplier theorems, Radon transforms, and the Kakeya problem, to name a few. The seminal article by Nagel, Stein and Wainger \cite{NSW} contains a proof that the projection on $S^{n-1}$ of the set of directions
\[
\Omega_\lambda\coloneqq \{(\lambda^{k\alpha_1},\ldots,\lambda^{k\alpha_n}):\, k\geq 1\},\qquad 0<\lambda<1,
\]
gives rise to a bounded maximal operator $\M_{\Omega_\lambda}$ in any dimension $n\geq 2$. Besides providing the first higher dimensional example of such a set of directions, the article \cite{NSW} contains the important novelty of treating the geometric maximal operator $\M_\Omega$ through Fourier analytic tools. This allowed the authors to break the barrier $p=2$ that was present in previous work of C\'ordoba and Fefferman  \cite{CorFef}, and Str\"omberg   \cite{Stromberg}, where the authors used mostly geometric arguments.  

Sj\"ogren and Sj\"olin \cite{SS} proved that, in dimension $n=2$, a sufficient condition for the $L^p$-boundedness of $\M_\Omega$ for some (equivalently for all) $1<p<\infty$ is that $\Omega$ is a lacunary set of finite order; loosely speaking, in dimension $n=2$, a  lacunary set $\Omega'$ of order $L$ is obtained from a lacunary set $\Omega$ of order $L-1$ by inserting within each gap between two consecutive elements $a,b\in \Omega$ two subsequences of suitably rotated copies of  $\Omega^{1}$ having $a,b$ as limit points.

Bateman \cite{BAT} subsequently showed that (up to finite unions) finite order lacunarity of $\Omega$ is necessary in order for $\M_\Omega$ to admit nontrivial $L^p$-bounds when $\Omega$ is an infinite set. While the counterexample by Bateman is highly nontrivial and employs a probabilistic construction based upon   tree percolation, it is rather easy to see that the $L^p$ norm of $\M_\Omega$  must depend on $N$ if $\Omega$ is, say, the set of $N$-th roots of unity. In fact, the sharp dependence  
\begin{equation}
\label{eq:sharpM}
\|\M_\Omega\|_{L^2(\R^2)\to L^2(\R^2)}\sim \log N
\end{equation}
 for sets of this type was proved by  Str\"omberg \cite{Stromberg}; in \cite{Katz}, the structural restriction on $\Omega$ was lifted and the upper bound in \eqref{eq:sharpM} was shown to hold for all finite $\Omega\subset S^1$ with cardinality $N$. Further results concerning   maximal operators along directions coming from sets with intermediate Hausdorff and fractal dimension can be found in \cite{PH,Hare}.

We already reviewed that in all dimensions $n\geq 2$, \cite{NSW} provides us with an example of a lacunary set of directions $\Omega_\lambda$ for which $M_{\Omega_{\lambda}}$ is bounded on $L^p(\R^n)$ for all $1<p<\infty$. Another significant higher dimensional example is given in the article of Carbery, \cite{Carbery}, where the author considers the projection on $S^{n-1}$ of the infinite set 
\[
\Omega^{n-1}\coloneqq \left\{(2^{k_1},\ldots,2^{k_n}):(k_1,\ldots,k_n)\in \mathbb Z\right\},
\] 
and proves that $\M_{\Omega^{n-1}}$ is bounded on $L^p(\R^n)$ for all $1<p<\infty$. In dimension $n=2$ the set $\Omega^{1}$ is the paradigmatic example of lacunary subset of $S^1$.  By the same token, the Carbery set $\Omega^{n-1}$ can be considered as the canonical example of a higher order, higher dimensional lacunary set. More precisely (cf. Definition~\ref{def:lacunary})   $\Omega^{n-1}$ is  a  lacunary set of order $n-1$  with exactly one direction in each cell of the dissection.

In dimensions $n\geq 3$, a general sufficient  characterization of those infinite  $\Omega\subset S^{n-1}$ giving rise to bounded directional maximal operators,  subsuming those of \cite{Carbery,NSW}, was recently established by Parcet and Rogers \cite{PR2} via an almost-orthogonality principle for the $L^p$ norms of $\M_{\Omega}$, resembling in spirit that of \cite{ASV} by Alfonseca, Soria and Vargas in the two-dimensional case. This principle leads naturally to the notion of a lacunary subset of $S^{n-1}$ when $n\geq 3$, which is a sufficient condition for nontrivial $L^p$-bounds of \eqref{eq:MF}. Again loosely speaking, $\Omega\subset S^{n-1}$ is lacunary  of order 1 if there exists a choice of orthonormal basis-- in the language of \cite{PR2}, a \emph{dissection} of the sphere $S^{n-1}$-- such that for all pairs of coordinate vectors $e_j,e_k$ the projection of $\Omega$ on the linear span of $e_j,e_k$ is a two-dimensional lacunary set; higher order lacunary sets are defined inductively in the natural way. The authors of \cite{PR2} also provide a necessary condition which is slightly less restrictive than finite order lacunarity; we send to their article for a precise definition. 

As \eqref{eq:Karaglb} shows, if $\Omega\subset S^{n-1}$ is  infinite, $H_{\Omega}$ is necessarily unbounded. Therefore, the question of sharp quantitative bounds for $H_\Omega$ in terms of the (finite) cardinality of the set $\Omega$ arises as a natural substitute of uniform bounds. In dimension $n=2$, several sharp or near-sharp results of this type have been obtained by Demeter \cite{Dem}, Demeter and the first author \cite{DDP}, and the authors \cite{DPP}. We choose to send to these references for detailed statements and just mention the quantitative bounds which are the closest precursors of our Theorem \ref{th:mainintro}. To begin with, the two-dimensional analogue of \eqref{eq:th:mainintro}, with the same $O(\sqrt{\log N})$ quantitative dependence, was proved by the authors in \cite{DPP}.  The methods of  \cite{DPP} are essentially relying on the fact that (lacunary) directions in $n=2$ can be naturally ordered, and that this order yields a telescopic representation of $H_\Omega$ as a maximal  partial sum of  Fourier restrictions to disjoint (lacunary) cones: see also \cite{Dem,Karag}.
These methods do not extend to dimensions three and higher, where no ordering is possible in general.  
On the other hand, in \cite[Corollary 4.1]{PR2}, following the ideas of \cite[Theorem 1]{DDP}, the authors derive the quantitative estimate \begin{equation}
\label{eq:Hunt}
\sup_{\substack{O\subset \Omega \\ \# O\leq N}} \left\|H_{O}\right\|_{L^p(\R^n)\to L^p(\R^n)} \leq C \log N \left\|\M_{\Omega}\right\|_{L^p(\R^n)\to L^p(\R^n)}
\end{equation}
at the root of which lies Hunt's exponential good-$\lambda$ comparison principle between maximal and singular integrals \cite{Hunt}. Coupling \eqref{eq:Hunt} with the main result of \cite{PR2} yields that the norms of the maximal Hilbert transform over finite subsets of a given finite order lacunary set \emph{in any dimension} grow at most logarithmically with the cardinality of the subset. When $n=3$, Theorem \ref{th:mainintro} improves this result to the sharp $O(\sqrt{\log N})$ quantitative dependence, answering the question posed by Parcet and Rogers in \cite[Section 4]{PR2}.  The estimate of Theorem \ref{th:mainintro} appears to be the first sharp quantitative estimate for directional singular integrals in dimension $n\geq 3$.

\subsection{Techniques of proof}  The key observation leading to    the  Parcet-Rogers theorem \cite{PR2}  is  that  the Fourier support of the \emph{single scale} distribution 
\[
f \mapsto \int_\R f(x-t \omega) \psi(t) \, \d t,
\] 
where $\psi$ is a  Schwartz  function on $\R$, is   covered by a union of two dimensional wedges $\Psi_{\sigma,\omega}$ over pairs $\sigma$ of coordinate directions, provided a suitable smooth $n$-dimensional average of $f$ at the same scale  is subtracted off; the latter piece is controlled by the strong maximal function of $f$. While these wedges  heavily overlap with respect to $\sigma$, see e.g.\ Figure \ref{fig:wedges}, the authors use the inclusion-exclusion principle to reduce to a square function estimate for compositions of two-dimensional multipliers adapted to the wedges $\Psi_{\sigma,\omega}$. The fact that this square function is a bounded operator on $L^p$ follows from the bounded overlap, for fixed $\sigma$, as $\omega$ ranges over a lacunary set $\Omega$, of the associated wedges $\Psi_{\sigma,\omega}$. The proof of our Theorem \ref{th:mainintro} is  also based on a representation of the directional Hilbert transform $H_\omega$ involving  two-dimensional wedge multipliers, which splits $H_\omega$ into an \emph{inner} and an \emph{outer} part: cf.\ Lemma \ref{lemma:represent}.  

The inner part, which is supported on the union of the wedges $\Psi_{\sigma,\omega}$, is amenable to a square function treatment; however, additional difficulties are encountered in comparison to \cite{PR2} as $H_\Omega$ is not a positive operator and does not obey a trivial $L^\infty$-estimate. We circumvent this difficulty by aiming for the stronger $L^2$-weighted norm inequality and relying on extrapolation theory for suitable weights in the natural directional $A_2$ classes. This requires extending the maximal inequality of \cite{PR2} to the weighted setting; while this extension does not require substantial additional efforts  we wrote out the proofs in detail for future reference. As we previously remarked, it also has the pleasant effect of giving a much more general weighted version of Theorem \ref{th:mainintro}: see Corollary \ref{cor:main}, Section \ref{sec:aop}.

Unlike the single scale operator, the outer part of the decomposition is nontrivial, and is actually the one introducing the dependence on the cardinality of the set of directions. It is a signed sum of $2^n$ terms which are compositions of two-dimensional angular multipliers; in general we cannot do better than  estimating the maximal operator associated to each summand. The key observation of our analysis at this point is that these compositions can be bounded pointwise by (compositions of)  strong maximal operators, upon pre-composition with at most $\lfloor \frac n 2 \rfloor  $ directional Littlewood-Paley projections; see Lemma \ref{lemma:annular} and Remark \ref{rem:cex2}. An application of at most $\lfloor \frac n 2 \rfloor  $ Chang-Wilson-Wolff decouplings, see Proposition \ref{prop:weightedcww}, then reduces the maximal estimate to a square function estimate upon loss of  $\lfloor \frac n 2 \rfloor  $ factors of order $\sqrt{\log N}$. This is enough to obtain the sharp result for $n=2,3$ (and, less interestingly, recover \eqref{eq:Hunt} when $n=4,5$), hinting  on the other hand that this approach is not feasible in general dimensions. 

In fact, perhaps surprisingly, we show with a counterexample that this growth rate, worse than that of $H_\Omega$ whenever $n\geq 6$, is actually achieved by the maximal operator associated to the outer parts. This phenomenon displays how the model operator of Lemma \ref{lemma:represent}, based on the combinatorics of two-dimensional wedges, is not subtle enough to   completely capture the cancellation present in $H_\Omega$.

%%%%%%%%%%%%%%%%%%%%%%%%%%%%%% SECTION SECTION SECTION
 
\subsection{Relation to the Hilbert transform along vector fields}
In addition to their intrinsic interest,   Theorem \ref{th:mainintro} and  predecessors may be seen as building blocks towards the resolution of the following question, apocryphally  attributed to E.\ Stein and often referred to as \emph{the vector field problem}: if $v:\R^n\to S^{n-1}$ is a vector field with     Lipschitz constant equal to 1 and pointing within a small neighborhood of $(1/{\sqrt{n}},\ldots,1/{\sqrt{n}})$, prove or disprove that the truncated directional Hilbert transform along $v$
\[
H_{v}f(x) =   \mathrm{p.v.} \int_{|t|<\eps_0} f(x-tv(x)) \, \frac{\d t}{ t}
\]
for $\eps_0>0$ small enough, is a bounded operator from $L^{2}(\R^n)$ into $L^{2,\infty}(\R^n)$. The partial progress in dimension $n=2$, beginning with the work of Lacey and Li \cite{LacLi:mem,LacLi:tams} and continued in e.g.\ \cite{BatThiele,Guo2,DPGTZK} by several authors, rests upon using the Lipschitz property to achieve decoupling of the full maximal operator into a Littlewood-Paley square function similar in spirit to the one appearing in \eqref{eq:pw4}. The estimation of  a single Littlewood-Paley piece in the vector field case is more difficult than the pointwise estimate available to us in Lemma \ref{lemma:annular} and involves, in dimension $n=2$, time-frequency analysis of roughly the same parametric complexity as of that appearing in the Lacey-Thiele proof of Carleson's theorem \cite{LTC}. Lemma   \ref{lemma:annular} in  this context may be interpreted as a \emph{single tree estimate} (cf.\ \cite{LTC,LacLi:tams}), showing that the annular estimate for  $n=3$ might display the same essential complexity as the $n=2$ case.

%%%%%%%%%%%%%%%%%%%%%%%%%%%%%% SECTION SECTION SECTION
 \subsection{Plan of the article} In the forthcoming Section \ref{sec:lacsetup}, we set up the notation for the remainder of the article and provide the precise definition of finite order lacunary sets in $\R^n$.   Section \ref{sec:model} contains the reduction of $H_\Omega$ to the above mentioned model operators, Lemma \ref{lemma:represent} as well as their \emph{single tree estimate} of Lemma \ref{lemma:annular}. In Section  \ref{sec:wnidmo}, after the necessary setup for directional weighted classes,  we prove a weighted version of the Parcet-Rogers maximal estimate in Theorem \ref{thm:weightedM} which, together with the extrapolation techniques of Lemma \ref{lemma:extr1}, is relied upon in the proof of our main result. Theorem \ref{th:mainintro} is derived in Section \ref{sec:aop} as the Lebesgue measure case of a more general sharp weighted estimate, Corollary \ref{cor:main}. This corollary in turn descends from Theorem \ref{th:main},  a $L^2$-weighted almost-orthogonality principle for $H_\Omega$ in the vein of \cite{ASV,PR2}. The final Section \ref{sec:cex} contains the above mentioned sharp counterexamples for the model operator of Lemma \ref{lemma:represent} in dimension 4 and higher: the main result of this section is the lower bound of  Theorem \ref{thm:cex}.

%%%%%%%%%%%%%%%%%%%%%%%%%%%%%% SECTION  SECTION SECTION
\subsection*{Acknowledgments} The authors are deeply grateful to Sara Maloni for fruitful discussions on the subject of completion of a lacunary set. We would also like to thank Maria J. Carro for helpful discussion related to weighted norm inequalities for  directional operators. We are indebted to Keith Rogers for an expert reading and insightful comments that helped us improve the presentation. Finally, we would like to thank the anonymous referees for providing helpful comments and references.

%%%%%%%%%%%%%%%%%%%%%%%%%%%%%% SECTION SECTION SECTION
\section{Lacunary sets of directions: definitions and notation}\label{sec:lacsetup}
In this section, we  give a rigorous definition of finite order lacunary sets which will be used throughout the article. In   essence, our definition is the same as the one given by Parcet and Rogers in  \cite{PR2}.  

%%%%%%%%%%%%%%%%%%%%%%%%%%%%%% SECTION  SECTION SECTION
\subsection{Lacunary sets of directions of finite order}
 For convenience we keep most of the notational conventions of \cite{PR2}. Throughout the paper we work in $\R^n$ and consider sets of directions $\Omega\subset S^{n-1}$. We allow the possibility that $\mathrm{span}(\Omega)=\R^d$ for some non-negative integer $d\leq n$ and write $\Sigma(d) \coloneqq \{	(j,k): \, 1\leq j<k\leq d\}$; we will drop the dependence on $d$ and just write $\Sigma$ when there is no ambiguity. We typically denote the members of $\Omega$ as ${\omega}$ and the members of $\Sigma$ as $\sigma=(j,k)$. Note that $|\Sigma(d)|=d(d-1)/2$.

With the roles of $n,d$, and $\Omega$ as above we assume that for each $\sigma=(j,k)\in\Sigma(d)$ we are given a sequence $\{\theta_{\sigma,\ell}:\, \ell \in \mathbb Z \}$ with the property that there exists $\lambda_\sigma\in(0,1)$ such that
\begin{equation}
\label{lacseq}  \theta_{\sigma,\ell+1} \leq\lambda_\sigma  \theta_{\sigma,\ell},\qquad \theta_{\sigma,0}=\theta_0,\quad\forall \sigma.
\end{equation}
Here we set $\lambda\coloneqq \max_{\sigma\in\Sigma}\lambda_\sigma$ and throughout the paper we will fix a numerical value of $\lambda\in(0,1)$ and we will adopt the convention that all sequences $\theta_{\sigma,\lambda}$ have lacunarity constants uniformly bounded by the same number $\lambda$.
A choice of orthonormal basis (ONB)  of $\mathrm{span}(\Omega)\equiv\R^d$\begin{equation}
\label{eq:basis}
\mathcal B\coloneqq\{ {e}_j:\, j =1,\ldots,d\}
\end{equation}
 and of 
 lacunary sequences $\{\theta_{\sigma,\ell}\}$ as above induces for each  $\sigma \in \Sigma(d)$ a partition of the sphere $S^{d-1}$ into sectors $S_{\sigma,\ell}$:
\begin{equation}
\label{lacdiss}
S^{d-1} =  \bigcup_{\ell \in \mathbb Z} S_{\sigma,\ell},\qquad S_{\sigma,\ell}=S_{(j,k),\ell}\coloneqq \left\{{\omega} \in S^{d-1} :\, \theta_{\sigma,\ell+1}\leq \frac{|{\omega}\cdot  {e}_k|}{|{\omega}\cdot  {e}_j|} <\theta_{\sigma,\ell} \right\}.
\end{equation}
We will henceforth write $\omega_j\coloneqq {\omega}\cdot  {e_j}$ for $1\leq j\leq d$ once the coordinate system is clear from context. The partition above is completed by adding the set $S_{\sigma,\infty}=S_{(j,k),\infty}\coloneqq S^{d-1} \cap ( {e}_j ^\perp\cup{e}_k ^\perp)$. We henceforth write $\mathbb Z ^*\coloneqq \mathbb Z \cup\{\infty\}$.
Now such a partition of the sphere immediately gives a partition of $\Omega$ by setting \begin{equation}
\label{eq:defsectors}
\Omega_{\sigma,\ell}\coloneqq \Omega\cap S_{\sigma,\ell},
\qquad 
\sigma \in \Sigma, \, \ell \in \mathbb Z^*.
\end{equation}  The family of ${d\choose 2}= d(d-1)/2$ partitions indexed by $\sigma\in\Sigma(d)$ ,
\[
\Omega=\bigcup_{\ell \in \mathbb Z} \Omega_{\sigma,\ell},
\] 
will be called \emph{a lacunary dissection} of $\Omega$, with parameters an ONB   $\mathcal B$ as in \eqref{eq:basis} and a choice of sequences $\{\theta_{\sigma,\ell}\}$ as in \eqref{lacseq}. Note that $\big\{\{S_{\sigma,\ell}\}_{\ell\in\mathbb Z^*}:\, \sigma\in\Sigma(d)\big\}$ is a lacunary dissection of $S^{d-1}$. 

 We will refer to sets of the type  $S_{\sigma,\ell}$ and $\Omega_{\sigma,\ell}$ as \emph{sectors} of the lacunary dissection.
We will also work with the partition of $\Omega$ into disjoint \emph{cells} induced by a dissection, namely  intersections of sectors $\Omega_{\sigma,\ell_\sigma}$. 
More precisely, let $\mathcal B$ be a choice of ONB as in \eqref{eq:basis}.
Given ${\cic{\ell}} =\{\ell_{\sigma}: \sigma \in \Sigma(d)\}\in{\mathbb Z} ^\Sigma$ we define the ${\cic{\ell}}$-cell of the dissection  corresponding to $\mathcal B$ as
\[
S_{{\cic{\ell}}}\coloneqq\bigcap_{\sigma \in \Sigma} S_{\sigma, \ell_\sigma},\qquad \Omega_{{\cic{\ell}}} \coloneqq \bigcap_{\sigma \in \Sigma} \Omega_{\sigma, \ell_\sigma}.
\]
Observe that this provides the finer partition of $S^{d-1}$ and $\Omega$, respectively, into \emph{cells}
\[
S^{d-1}=\bigcup_{{\cic{\ell}}\in \mathbb {Z}^\Sigma}S _{{\cic{\ell}}}, \qquad \Omega = \bigcup_{{\cic{\ell}}\in{\mathbb Z}^\Sigma} \Omega_{{\cic{\ell}}}.
\]

The following definition, which is the principal assumption in our main results, was given in \cite[p.\ 1537]{PR2}.

%%%%%%%%%%%%%%%%%%%%%%%%%%%%%% DEFINITION DEFINITION DEFINITION
\begin{definition}[Lacunary set]\label{def:lacunary} Let $\Omega\subset S^{n-1}$ be a set of directions with $\mathrm{span}(\Omega)=\R^d$. Then 
\begin{itemize}
\item[$\cdot$]
$\Omega$ is a lacunary set of order $0$ if it consists of a single direction;
\item[$\cdot$] if $L$ is a positive integer, then $\Omega$ is  lacunary of order $L$ if there exists an ONB     $\mathcal B$ as in \eqref{eq:basis} and a choice of sequences $\{\theta_{\sigma,\ell}\}$ as in \eqref{lacseq} with the property that for each $\sigma\in\Sigma(d)$ and each $\ell\in\mathbb Z^*$ the sector $\Omega_{\sigma,\ell}$ in \eqref{eq:defsectors}  is a lacunary set of order $L-1$.
 \end{itemize}
A set $\Omega$ will be called \emph{lacunary} if it is a finite union of lacunary sets of finite order.
\end{definition}
%%%%%%%%%%%%%%%%%%%%%%%%%%%%%% DEFINITION DEFINITION DEFINITION

For example, $\Omega$ is $1$-lacunary if there exists a dissection such that, for each $\sigma\in\Sigma(d)$ and $\ell\in \N$ the set $\Omega_{\sigma,\ell}$ contains at most one direction.

%%%%%%%%%%%%%%%%%%%%%%%%%%%%%% REMARK REMARK REMARK
\begin{remark}\label{rem:2-j} Let $\Omega$ be a lacunary set of directions and $\beta\in(0,1)$. Then $\Omega$ is a lacunary set of directions with respect to dissections given by the sequence $\theta_{\sigma,\ell}\coloneqq \beta^{\ell}$. This is automatic if $\beta\geq\lambda$ while in the case $\beta<\lambda$ it follows easily by suitably splitting the set $\Omega$ into $O(\log\beta/\min_\sigma\log\lambda_\sigma)$ congruence classes. Unless explicitly mentioned otherwise, all lacunary sets in this paper are given with respect to the sequence
\[
 \theta_{\sigma,\ell}={2^{-\ell}}, \qquad {\ell\in\mathbb Z}.
\]
As our choice of sequences $ \{\theta_{\sigma,\ell}\}$ is universal,  prescribing a lacunary dissection amounts to fixing an orthonormal basis $\mathcal B$ as in \eqref{eq:basis}.
It is also clear that for all proofs in this paper it suffices to consider the case that $\Omega$ is contained in the open positive $2^d$-tant of the sphere $S^{d-1}_+\coloneqq S^{d-1}\cap \R^d _+$. While there are different coordinate systems involved in the definition of a lacunary set $\Omega$, by splitting any lacunary set into finitely many pieces we can assume this property for all dissections that come into play. Furthermore, by standard approximation arguments (e.g. monotone convergence) we can assume that $\Omega$ has empty intersection with all coordinate hyperplanes. These conventions allow us to only consider sectors $S_{\sigma,\ell},\Omega_{\sigma,\ell}$ with $\ell \in \mathbb Z$ instead of $\ell \in \mathbb Z^*$.

On the other hand, and in contrast with the previous conventions concerning the proofs, in the statements of our theorems we always assume that the set $\Omega$ is closed. Furthermore, the basis vectors of any dissection used in the definition of a lacunary set of any order are assumed to be contained in the set. We adopt these conventions throughout the paper without further mention. 
\end{remark}
%%%%%%%%%%%%%%%%%%%%%%%%%%%%%% REMARK REMARK REMARK

%%%%%%%%%%%%%%%%%%%%%%%%%%%%%% REMARK REMARK REMARK
\begin{remark} Although it is necessary to distinguish the case $\mathrm{span}\Omega=\R^d$ with $d<n$ in the definitions,  in the proofs of our estimates we will argue with $d=n$ without explicit mention; by Fubini's theorem, this is without loss of generality.
\end{remark} 
%%%%%%%%%%%%%%%%%%%%%%%%%%%%%% REMARK REMARK REMARK

%%%%%%%%%%%%%%%%%%%%%%%%%%%%%% SECTION  SECTION SECTION
\section{Model operators}  \label{sec:model} For ${\omega}\in S^{n-1}$  (re)define the directional Hilbert transform on
 $\R^n$
\begin{equation} \label{eq:Tomega}
H_{{\omega}} f(x) = \int_{\R^n} \widehat{f}(\xi) \mathrm{sign}  (\xi \cdot {\omega}) \e^{ix\cdot \xi }\, \d \xi .
\end{equation}
In this section we set up a representation formula for \eqref{eq:Tomega}. The central result is Lemma \ref{lemma:represent} below. 
Before the statement we need to introduce some additional notation and auxiliary functions.  For $\ell \in \mathbb Z$ and $\gamma>0$ we consider the two-dimensional wedges
\[
\Psi_{\sigma,\ell,\gamma}\coloneqq\left\{\xi \in \R^n\setminus ({e}_{\sigma(2)})^\perp : \frac{2^{-(\ell+1)}}{\gamma} \leq -\frac{\xi_{\sigma(1)}}{\xi_{\sigma(2)}} <  \gamma 2^{-\ell}  \right\}.
\]
We are interested in the particular cases $\gamma\in \{n,n+1\}$  for which we use the special notations
\begin{equation}
\label{eq:wedgesp}
\Psi_{\sigma,\ell,n}\eqqcolon \Psi_{\sigma,\ell}, \qquad  \Psi_{\sigma,\ell,n+1} \eqqcolon \widetilde \Psi_{\sigma,\ell}.
\end{equation}
Furthermore, let $\phi^{+},\phi^{-}:\R \to [0,1]$ be  smooth functions satisfying
\[
\phi^{+}(x)\coloneqq \begin{cases} 0, & x<-(n+1),  \\ 1, & x>-n ,  \end{cases}\qquad 
\phi^{-}(x)\coloneqq \begin{cases} 1, & x<-\frac{1}{2n},   \\ 0, & x>- \frac{1}{2(n+1)}.  \end{cases}
\]
\begin{figure}[b] 
\begin{center}
\includegraphics[width=0.98\textwidth]{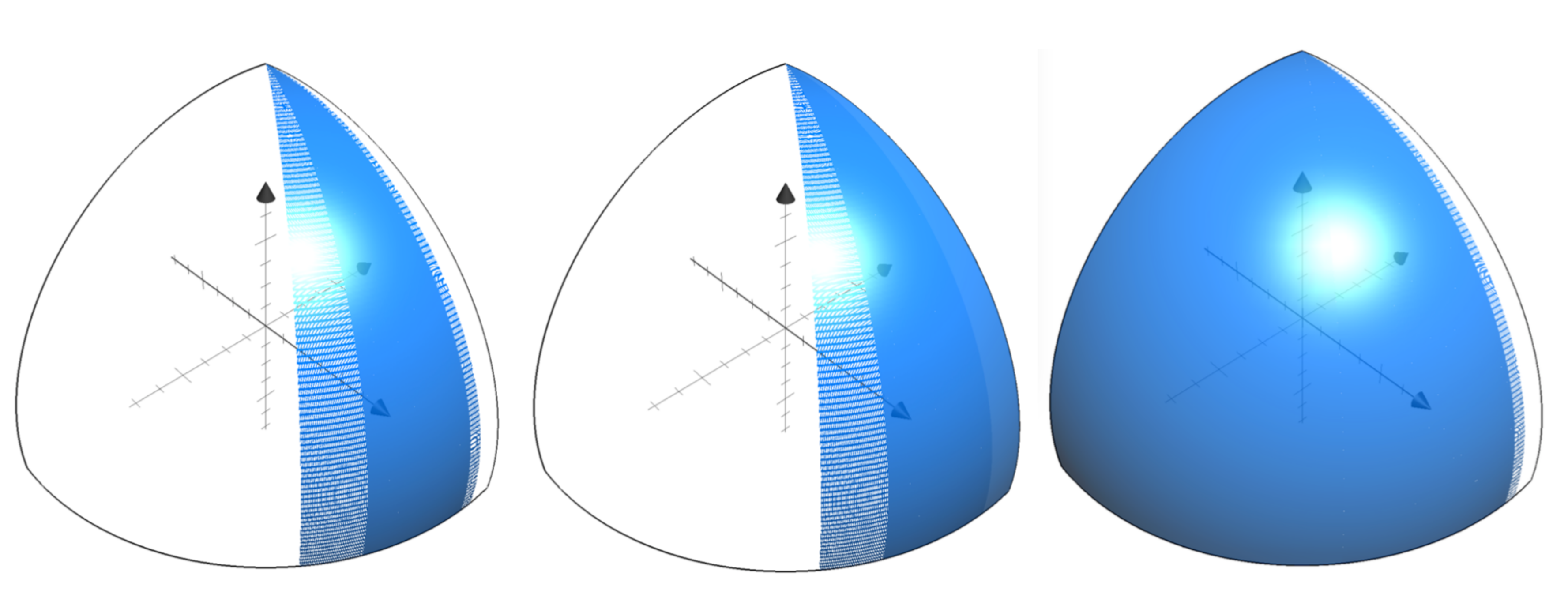}
\caption{\small  The Fourier support of the multipliers $K^{\circ}_{(1,2),\sigma},$ $K^{-}_{(1,2),\sigma},$ and $K^{+}_{(1,2),\sigma}$.\label{fig:multi}}
\end{center}
\end{figure}
We now use the functions $\phi^{+},\phi^{-}$ in order to define the essentially two-dimensional angular Fourier multiplier operators  
\begin{equation}
\label{eq:fmo} 
\begin{split} &
\widehat {K_{\sigma,\ell _\sigma}^{\pm}} (\xi)= \kappa_{\sigma,\ell _\sigma}^{\pm}(\xi_{\sigma(1)},\xi_{\sigma(2)})\coloneqq  \phi^{\pm} \left(2^{\ell_\sigma}\frac{\xi_{\sigma(1)}}{\xi_{\sigma(2)}}\right), 
\\ 
&\widehat {K_{\sigma,\ell _\sigma}^\circ } (\xi)= \kappa_{{\sigma,\ell _\sigma} }^\circ(\xi)\coloneqq  \kappa_{\sigma,\ell _\sigma}^{+}(\xi)\kappa_{\sigma,\ell _\sigma}^{-}(\xi),   \end{split}\end{equation} 
and their compositions
\begin{equation} \label{eq:KUelleps}
 K_{U,{\cic{\ell}}}^{{\cic{\eps}} }\coloneqq   \prod_{\sigma \in U}K_{\sigma,\ell _\sigma}^{\eps_\sigma}, \qquad \varnothing \subsetneq U\subseteq \Sigma,\qquad {\cic{\eps}} \in \{+,\circ,-\}^U;
\end{equation}
when $\eps_{\sigma}=\circ $ for all $\sigma \in U$ we simply write $K_{U,{\cic{\ell}}}$ in place of $K_{U,{\cic{\ell}}}^{{\cic{\eps}} }\,. $
\begin{figure}[h] 
\begin{center}
\includegraphics[width=0.45\textwidth]{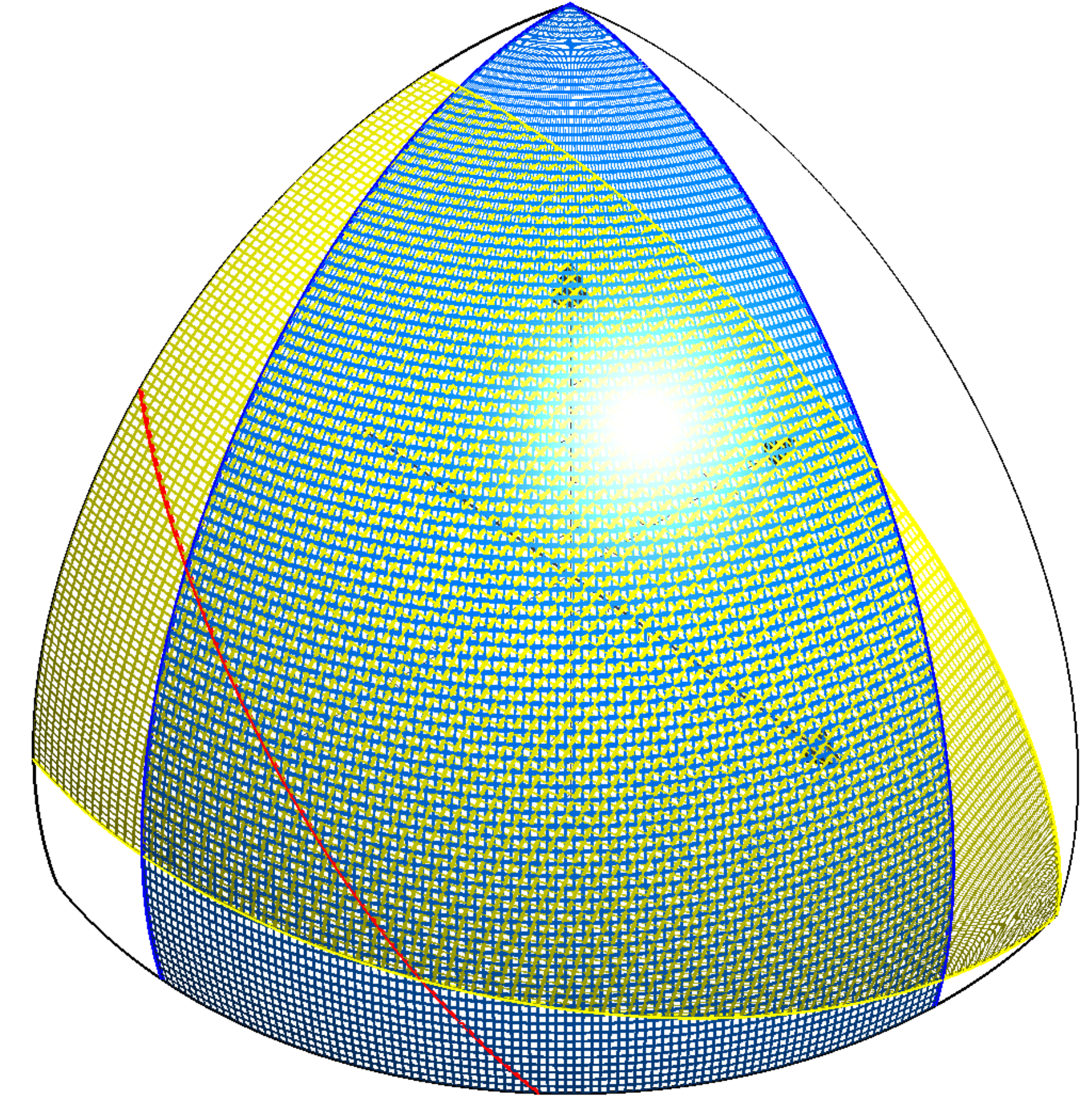}
\caption{\small  Suppose $\omega$ belongs to the cell $S_{\cic{\ell}}$. The red line is the intersection with the sphere $S^2$ of the singularity $\xi\cdot \omega=0$ of $H_\omega$. The blue and yellow wedges are respectively $\Psi_{(1,2),\ell_{(1,2)}}$ and $\Psi_{(2,3),\ell_{(2,3)}}$ from \eqref{eq:wedgesp}. As in the depicted octant $\xi_1$ and $\xi_3$ have the same sign, $\Psi_{(1,3),\ell_{(1,3)}}$ is not visualized. \label{fig:wedges}}
\end{center}
\end{figure}
%%%%%%%%%%%%%%%%%%%%%%%%%%%%%% REMARK REMARK REMARK
\begin{remark}\label{rem:IMPORTANT}  Let $\eps\in\{+,-,\circ\}$. We record the support conditions (see Figure \ref{fig:multi})
\begin{equation}
\label{eq:support}
\big( \nabla_\xi \kappa_{\sigma,\ell_\sigma}^{\eps}\big)  \cic{1}_{\Psi_{\sigma,\ell}}\equiv \big( \nabla_\xi \kappa_{\sigma,\ell_\sigma}^{\eps}\big)  \cic{1}_{\R^n \setminus \widetilde{\Psi}_{\sigma,\ell}}\equiv 0,   \qquad
\kappa^{\eps} _{\sigma,\ell _\sigma}\cic{1}_{\Psi_{\sigma,\ell}} \equiv 1, \qquad \kappa^\circ _{\sigma,\ell _\sigma}\cic{1}_{\R^n \setminus \widetilde{\Psi}_{\sigma,\ell}} \equiv 0.
\end{equation}
Moreover we have the derivative estimates
\begin{equation}
\label{eq:fmoest} \sup_{|\alpha|\leq 10n}
\sup_{\xi\in \R^n} |\xi_{\sigma(1)}|^{\alpha_1}|\xi_{\sigma(2)}|^{\alpha_2} \big| \partial_{\xi_{\sigma(1)}}^{\alpha_1} \partial_{\xi_{\sigma(2)}}^{\alpha_2} \kappa^\eps _{\sigma,\ell_{\sigma}}(\xi)\big| \lesssim 1, \qquad |\alpha|=\alpha_1+\alpha_2.
\end{equation}
We will also use below   that if $\xi\not \in \widetilde\Psi_{\sigma,\ell_\sigma}$, then $\kappa^\eps _{\sigma,\ell_\sigma}$ is constant  in a neighborhood of $\xi$.
\end{remark}
%%%%%%%%%%%%%%%%%%%%%%%%%%%%%% REMARK REMARK REMARK

%%%%%%%%%%%%%%%%%%%%%%%%%%%%%% LEMMA LEMMA LEMMA
\begin{lemma} \label{lemma:represent} Suppose ${\omega} \in S_{{\cic{\ell}}}$, the cell of $S^{n-1}$ with lacunary parameters ${\cic{\ell}}=\{\ell_\sigma:\sigma \in \Sigma\}$. Then we have the pointwise bound
\[
\begin{split}
 |H_{{\omega}} f|  &\lesssim |f| + \sup_{\varnothing \subsetneq U\subseteq \Sigma} \big|H_{{\omega}}   K_{U,{\cic{\ell}}}f \big|   + \sup_{{\cic{\eps}} \in \{+,-\}^\Sigma} \sup_{\varnothing \subsetneq U\subseteq \Sigma}  \big|   K_{U,{\cic{\ell}}}^{{\cic{\eps}} } f\big|.
\end{split}
\]
\end{lemma}
%%%%%%%%%%%%%%%%%%%%%%%%%%%%%% LEMMA LEMMA LEMMA

%%%%%%%%%%%%%%%%%%%%%%%%%%%%%% PROOF PROOF PROOF
\begin{proof} As
\[
\mathrm{Id} = \bigg[\sum_{\varnothing \subsetneq U\subseteq \Sigma}  (-1)^{\#U + 1} K_{U,{\cic{\ell}}} \bigg]+\bigg[ \prod_{\sigma \in \Sigma }\left( \mathrm{Id}- K_{\sigma,\ell _\sigma} \right)\bigg ]
\]
we write 
\begin{equation}
\label{eq:split1}
H_{{\omega}} f = \bigg[  \sum_{\varnothing \subsetneq U\subseteq \Sigma}  (-1)^{\#U + 1} H_{{\omega}} K_{U,{\cic{\ell}}}f \bigg] +T  f
\end{equation}
where $T$ is the Fourier multiplier with symbol
\begin{equation}
\label{eq:multi}
m(\xi)=
\widehat{T} (\xi)= \mathrm{sign} ({\omega}\cdot \xi) \prod_{\sigma \in \Sigma} \big(1-  \kappa_{\sigma,\ell _\sigma}^\circ(\xi)\big) .
\end{equation}
We have to treat the term $T$. First of all, we check that
\begin{equation} \label{eq:inclusion}
C_{{\omega}}\coloneqq\Big\{\xi \in \R^n : |\xi \cdot{\omega}|<\frac{1}{n}\max_{1\leq j\leq n}|\omega_j\xi_j|  \Big\} \subset D_{{\cic{\ell}}}\coloneqq\bigcup_{\sigma \in \Sigma} \Psi_{\sigma,\ell_\sigma}.
\end{equation}
This is essentially depicted in Figure \ref{fig:wedges} and is a sharpening of the argument in \cite[Proof of Theorem A]{PR2}. We prove \eqref{eq:inclusion} by showing that $\R^n\setminus D_{{\cic{\ell}}}  \subseteq  \R^n\setminus C_{{\omega}}$. To that end let $\xi\in \R^n\setminus D_{{\cic{\ell}}} $. Writing $\eta_j\coloneqq\omega_j \xi_j $ and remembering the convention $\omega_j>0$ for all $j$ we then have that
\[
-\frac{\eta_{\sigma(1)}}{\eta_{\sigma(2)}} \not \in \textstyle\left[\frac{1}{n} , n \right]  \qquad \forall \sigma \in \Sigma.
\]  
Choose $j_\star$ such that $|\eta_{j_\star}| =\max_{1\leq j\leq n} |\eta_j|$. Now we note that if $\eta_j \eta_{j_\star}\geq 0$ for all $ j\in\{1,\ldots,n\}\setminus \{j_\star\}$ then $\xi \notin C_{{\omega}}$ so we are done. Otherwise we define $k_\star$ by means of 
$ |\eta_{k_\star}|\coloneqq\max_{j: \eta_{j_\star}\eta_j<0} |\eta_j|;$ as $|\eta_{j_\star}|\geq n|\eta_{k_\star}|$ we end up with  
\[
|\xi \cdot{\omega}|=
\Big|\sum_{1\leq j\leq n} \eta_j\Big| \geq |\eta_{j_\star}| - (n-1) |\eta_{k_\star}| \geq \frac{|\eta_{j_\star}|}{n} = \max_{1\leq j\leq n}|\omega_j\xi_j|
\]
which is the claim \eqref{eq:inclusion}. Noting that \[\mathrm{supp}\, m =\R^n\setminus \bigcup_{\sigma\in\Sigma} \Psi_{\sigma,\ell_{\sigma}}=\R^n\setminus D_{{\cic{\ell}}}\]  this claim tells us that   $ \mathrm{supp}\, m \cap C_{{\omega}} =\emptyset$ whence if $\xi \in\mathrm{supp}\, m  $ the signum of $({\omega}\cdot \xi)$ is constant in a neighborhood of $\xi$. Now using the easy to verify fact that $(1-\phi^+\phi^{-}) = (1-\phi^+)+ (1-\phi^{-})  $ and  the two summand are supported in disjoint intervals we can rewrite \eqref{eq:multi} as
\[
m (\xi) =
\sum_{{\cic{\eps}} \in \{+,-\}^\Sigma}  \mathrm{sign} ({\omega}\cdot \xi) \kappa_{{\cic{\eps}}}(\xi), \qquad  \kappa_{{\cic{\eps}}}(\xi)\coloneqq  \prod_{\sigma \in \Sigma}\big(1-\kappa_{\sigma,\ell _\sigma}^{\eps_\sigma}\big),
\]
for $\cic{\eps}=\{\eps_\sigma:\, \sigma\in\Sigma\}$. As $\mathrm{supp} \,\kappa_{{\cic{\eps}}}  $ is a connected set not intersecting $C_{{\omega}}$ we conclude that $\mathrm{sign} ({\omega}\cdot \xi)$ is constant on $\mathrm{supp}\, \kappa_{{\cic{\eps}}}$. Therefore if $T_{{\cic{\eps}}}$ is the Fourier multiplier with symbol $\kappa_{{\cic{\eps}}}$
\begin{equation}
\label{eq:intermediate1}
|T f|\leq  \sum_{{\cic{\eps}}\in \{+,-\}^{\Sigma}} |T_{{\cic{\eps}}} f|.
\end{equation}
Now we observe that the symbol of $\mathrm{Id}-T_{{\cic{\eps}}}$ is equal to 
\[
 1-\prod_{\sigma \in \Sigma}\big(1-\kappa_{\sigma,\ell _\sigma}^{\eps_\sigma}\big) = \sum_{\varnothing\subsetneq U \subseteq \Sigma} (-1)^{\# U+1} \prod_{\sigma\in U} \kappa_{\sigma,\ell _\sigma}^{\eps_\sigma},
\]
and putting together the last display with \eqref{eq:split1} and \eqref{eq:intermediate1} we achieve  the pointwise estimate  claimed in  the Lemma.
\end{proof}
%%%%%%%%%%%%%%%%%%%%%%%%%%%%%% PROOF PROOF PROOF

In the next lemma  we prove an  annular estimate for the multiplier operators of \eqref{eq:KUelleps}. To do so we will need to precompose these operators with suitable Littlewood-Paley projections which we now define. Let $p,q$ be   smooth functions on $\R$ with
\[\begin{split} 
& \mathrm{supp} \, p\subset \left\{\xi \in \R :\, {\textstyle\frac12 < |\xi| <2}\right\}, \qquad \sum_{t\in \mathbb Z} p(2^{-t} \xi)=1,\quad \xi\neq 0   ,
\\ 
& \mathrm{supp} \, q  \subset \left\{\xi \in \R:\, {\textstyle \frac14 <} |\xi| <4\right\}, \qquad q=1\textrm { on }  \left\{\xi \in \R :\, {\textstyle\frac12 < |\xi| <2}\right\}. \end{split}
\]
Now for $\upsilon \in \{1,\ldots, n\}$ we define the Fourier multiplier operators on $\R^n$
\[
\widehat{P_t^{\upsilon} f}(\xi) \coloneqq  \widehat{ f}(\xi) p(2^{-t}\xi\cdot {e}_\upsilon), \qquad \widehat{Q_t^{\upsilon} f}(\xi) \coloneqq \widehat{ f}(\xi) q(2^{-t}\xi\cdot {e}_\upsilon).
\]
 Thus $\{P_t ^\upsilon\}_t$ is a one-dimensional   Littlewood-Paley decomposition, acting on the $\upsilon$-th variable only, and being the identity with respect to all other frequency variables. Here and in the rest of the paper we write $\mathrm M_{\mathsf{s}}$ for the strong maximal function and $\mathrm M^2 _{\mathsf{s}}\coloneqq \mathrm M_{\mathsf{s}}\circ \mathrm M_{\mathsf{s}}$. 
 
%%%%%%%%%%%%%%%%%%%%%%%%%%%%%% LEMMA LEMMA LEMMA 
\begin{lemma}\label{lemma:annular} Let $\mathrm{supp} \widehat{f} \subset Q$ where $Q$ is any of the $2^{3}$ octants of $\,\mathbb R^3$.  Let $\varnothing \subsetneq U\subseteq \Sigma,\, {\cic{\eps}} \in \{+,-\}^U$. There is a choice $\upsilon=\upsilon(U,{\cic{\eps}},Q)\in \{1,\ldots,n\}$ such that the pointwise estimate
\[
 \big|K_{U,{\cic{\ell}} \,}^{{\cic{\eps}} }(P_{t}^{\upsilon} f)(x)\big| \lesssim \mathrm M^2 _{\mathsf{s}} (P_{t}^{\upsilon} f) (x),\quad x\in\R^3,
\]
holds uniformly over all $t\in \R$.
\end{lemma}
%%%%%%%%%%%%%%%%%%%%%%%%%%%%%% LEMMA LEMMA LEMMA

%%%%%%%%%%%%%%%%%%%%%%%%%%%%%% PROOF PROOF PROOF
\begin{proof}  As ${\cic{\eps}}, {\cic{\ell}}$ are fixed throughout the proof, and in order  to avoid proliferation of indices, we shall write below
\[
\kappa_{\sigma,\ell_\sigma}^{\eps_\sigma}=\kappa_\sigma,\qquad  K_{\sigma,\ell_\sigma}^{\eps_\sigma}=K_\sigma,  \qquad K_{U,{\cic{\ell}}}^{{\cic{\eps}} } = K_U,
\]
when these parameters are unimportant. As we are working with the strong maximal function, by rescaling on the sphere we may assume  $  \ell_\sigma=0$ for all $\sigma \in \Sigma$; this is just for convenience of notation as we shall see. We divide the proof to different cases according to the cardinality of the set $U\subseteq \Sigma.$

\subsubsection*{Case $\#U=1$} In this case there exists $\sigma\in\Sigma(3)$ such that $U=\{\sigma\}$, $K_{U }=K_{\sigma}$, and we may choose either $\upsilon =\sigma(1)$ or $\upsilon=\sigma(2)$. The choice does not depend on the quadrant $Q$. To fix ideas, we work with $\sigma=(1,2)$ and   choose $\upsilon=1$.
By the observation \eqref{eq:support} of Remark \ref{rem:IMPORTANT}, we know that $ \kappa_{\sigma}  $ is constant in a neighborhood of $\xi$ unless $\xi \in \widetilde \Psi_{\sigma,0}$, in which case $|\xi_{\sigma(1)}|\sim |\xi_{\sigma(2)}| $. Therefore if $\xi\in\widetilde \Psi_{\sigma,0}$, and $2^{t-2}<|\xi_1|<2^{t+2}$,  there holds $2^t\sim|\xi_{\upsilon}|\sim |\xi_{\sigma(2)}| $ and 
\[
 \big| \partial_{\xi_{\sigma(1)}}^{\alpha_1} \partial_{\xi_{\sigma(2)}}^{\alpha_2} \kappa_{\sigma}(\xi)\big| \lesssim |{\xi_{\sigma(1)}}|^{-\alpha_1} | {\xi_{\sigma(2)}}|^{-\alpha_2}  \lesssim 2^{-t\alpha}, \qquad \alpha=\alpha_1+\alpha_2.
\]
Using the above inequality for $\alpha=0,\ldots, 10\cdot 3$, it follows that  
\begin{equation}
\label{eq:2dker}
\Phi_{\sigma}(x_{\sigma(1)},x_{\sigma(2)})\coloneqq \int_{\R^2} \kappa_{\sigma }(\xi_{\sigma(1)},\xi_{\sigma(2)}) q(2^{-t}\xi_{\upsilon}) \e^{i(x_{\sigma(1)}\xi_{\sigma(1)}+x_{\sigma(2)}\xi_{\sigma(2)})} \, \d \xi_{\sigma(1)}\d\xi_{\sigma(2)}
\end{equation}
satisfies
\begin{equation} \label{eq:Ksigma}  
|\Phi_{\sigma}(x_{\sigma(1)},x_{\sigma(2)})| \lesssim {2^{2t}}\left(  1 + 2^{t}|x_{\sigma(1)}| +2^{t}|x_{\sigma(2)}| \right)^{-(3+1)}.
\end{equation}
We now write $f_t=P_{t}^{\upsilon} f$. Denoting convolution in the variables $\sigma(1),\sigma(2)$ by  $*_{\sigma}$  we have that
\[
K_{\sigma }f_t= (K_{\sigma}Q_{t}^{\upsilon})(f_t) = \Phi_{\sigma}*_\sigma f_t.
\]
Hence using \eqref{eq:Ksigma} we see that
\[
|K_{U} f_t(x)| \leq \int_{\R^2} |f_t(x_1-y_1,x_2-y_2,x_3)| |\Phi_{\sigma}(y_1,y_2)| \, \d y \lesssim \mathrm M_{\mathsf{s}} ( f_t) (x)
\]
as claimed.

\subsubsection*{Case $\#U=2$} In this case $U=\{\sigma,\tau\}$ for some $\sigma,\tau\in\Sigma(3)$ and necessarily  $\sigma$, $\tau$ must have a common component. We choose $\upsilon$ to be this common component. This choice also does not depend on the quadrant $Q$. To fix ideas $\sigma=(1,2),\tau=(1,3)$ and we  choose  $\upsilon=1$. Note that in this case $K_{U}=K_{(1,2)}K_{(1,3)} $. With the same  notation of \eqref{eq:2dker} from  the previous case we have the equality 
\[
K_{U } f_t = \left(K_{(1,2)} Q_{t}^{\upsilon}\right) \circ \left(K_{(2,3)} Q_{t}^{\upsilon}\right) (f_t)= \Phi_{(1,2)} *_{(1,2)} \Phi_{(1,3)} *_{(1,3)} f_t
\]
so   using \eqref{eq:Ksigma} again we see that
\[
\begin{split}
|K_{U }  f_t(x)| &\leq \int_{\R^2\times \R^2} |f_t(x_1-y_1-z_1,x_2-y_2,x_3-z_3)| |\Phi_{(1,2)}(y_1,y_2)| |\Phi_{(1,3)}(z_1,z_3)|\, \d y  \d z \\ & \lesssim \mathrm M^2 _{\mathsf{s}} ( f_t) (x)
\end{split}
\]
as claimed.

\subsubsection*{Case $\#U=3$}
We show that this case reduces to the preceding ones, with choice of $\upsilon$ depending on the quadrant $Q$. Let $$Q_{\sigma}=\{\xi \in \R^3:\,\xi_{\sigma(1)} \xi_{\sigma(2)}\geq 0\}.$$
Notice that the constraints on the supports of $\phi^\pm$ imply that
\[
\kappa_{\sigma,\ell_\sigma  }^-\cic{1}_{Q_{\sigma}} \equiv 0,  \qquad \kappa_{\sigma,\ell_\sigma}^+\cic{1}_{Q_{\sigma}} \equiv 1, \qquad \forall \sigma \in \Sigma.
\]
As for each of the 8 quadrants $Q$ of $\R^3$ there exists (at least one) $\sigma_Q \in \Sigma$ such that $Q\subset Q_{\sigma_Q}$, we see that
\[
K_{U,{\cic{\ell}}}^{{\cic{\eps}} }\,\cic{1}_{Q} =\begin{cases}0,& \quad\text{if}\quad \exists\sigma\in U\quad\text{with}\quad \eps_\sigma=-, 
\\
K_{U\setminus \{\sigma_Q\},{\cic{\ell}}}^{{\cic{\eps}} }\,\cic{1}_{Q},&\quad{\text{otherwise}}.
\end{cases}
\]
As $\#\{U\setminus\{\sigma_Q\}\}=2$ for each quadrant $Q$ the proof follows by the cases $\#U\in\{1,2\}$ considered above.
\end{proof} 
%%%%%%%%%%%%%%%%%%%%%%%%%%%%%% PROOF PROOF PROOF

%%%%%%%%%%%%%%%%%%%%%%%%%%%%%% SECTION SECTION SECTION
\section{Weighted norm inequalities for directional maximal operators} \label{sec:wnidmo} We dedicate this section to the discussion of weighted norm inequalities for the  maximal  directional operator. These will serve as a tool for the proof of Theorem \ref{th:mainintro}; in fact, they will be used to prove a weighted almost orthogonality principle that subsumes both  Theorem \ref{th:mainintro} and its weighted analogue, which will be stated at the end of this section. However, we do think they are also of independent interest. 

The  weighted theory of the directional maximal operator has been studied, at least in the two-dimensional case, in \cite{DuoMoy}, for the case of 1-lacunary sets of directions. Here we recall all the basic definitions and tools, and then proceed to prove weighted norm inequalities for the directional maximal function $M_\Omega$ associated to a finite order lacunary set $\Omega\subset S^{n-1}$. In essence, the main result of this section, Theorem \ref{thm:weightedM}, is a weighted generalization of the main result of \cite{PR2} by Parcet and Rogers.
%%%%%%%%%%%%%%%%%%%%%%%%%%%%%% SECTION  SECTION SECTION
\subsection{Directional $A_p$ weights} We begin by defining the appropriate directional $A_p$ classes. The easiest way to define the appropriate class is to ask for non-negative, locally integrable functions $w$ (we will refer to such functions as weights) such that for all nice functions $f$ we have
\[
\|\M_{\Omega}f \|_{L^p(w)}\lesssim \|f\|_{L^p(w)},\qquad \|f\|_{L^p(w)}\coloneqq \Big(\int|f|^pw\Big)^\frac{1}{p},\qquad 1<p<\infty,
\]
where $\Omega$ is a set of directions such that $\M_{\Omega}$ is bounded on $L^p(\R^n)$. Without explicit mention, we work under the purely qualitative assumptions that all weights appearing below will be continuous  and nonvanishing functions on $\R^n$; this assumption may be removed via a standard approximation procedure which we omit.
 We will very soon specialize to sets $\Omega$ which are lacunary of finite order so we encourage the reader to keep this example in mind. Note that for smooth functions $f$ we have $\M_{\widebar{ \Omega}} f = \M_{\Omega}f $. We can then assume that $\Omega$ is closed when deriving necessary conditions for $w$. 
 
For   ${\omega}\in\Omega$, $x\in\R^n$ and $\eta>0$ we then define segments and corresponding one-dimensional averages of $ f\in \mathcal C(\R^d) $ as follows

\[
I({x},\eta,{\omega})\coloneqq\{{x}+ t{\omega}:\,|t|<\eta\},\qquad \l f \r_{I({x},\eta,{\omega})} \coloneqq \frac{1}{2\eta} \int_{-\eta}^\eta f(  x+ t{\omega}) \, \d t. 
\]
We set  $\I_\Omega\coloneqq \{I({x},\eta,{\omega}):\, {x}\in \R^d,\, \eta>0,{\omega}\in \Omega\}$ and for $p\in(1,\infty)$ we adopt the usual notation for the dual weight $\sigma\coloneqq w^{-\frac{1}{p-1}}$. Now the $L^p(\R^n)$-boundedness of $\M_\Omega$ clearly implies the boundedness of $\M_{{\omega}}$ on $L^p(I({x},{\omega}))$, where $I({x},{\omega})\coloneqq \{{x}+ t{\omega}:\,t\in\R\}$, uniformly in $x\in\R^n$ and ${\omega}\in\Omega$. Now testing this one-dimensional boundedness property $\M_{{\omega}}$ for some fixed $p\in(1,\infty)$ against functions of the form $\sigma\ind_{I({x},\eta,{\omega})}$ shows the necessity of the directional $A_p$ condition
\[
[w]_{A_p ^\Omega}\coloneqq\sup_{I\in\I_\Omega}\int_I w \, \Big(\int_I \sigma\Big)^{p-1}<\infty;
\]
here we remember that we have made the qualitative assumption that $w$ is a continuous non-vanishing function.

Note that if we write $w(x)=w(x\cdot {\omega},x\cdot{\omega} ^\perp)$, the previous condition means that for almost every $x\in\R^n$ and ${\omega}\in\Omega$, the one-dimensional weight $v_{x,{\omega}}(s)\coloneqq w(s,x\cdot {\omega}^\perp)$, $s\in\R$, is in $A_p(\R)$, with uniformly bounded $A_p$ constant:
\[
\sup_{x\in\R^n,\, {\omega}\in\Omega}[v_{x,{\omega}}]_{A_p}= [w]_{A_{p} ^\Omega}<\infty.
\]
We complete the set of definitions by defining $A_{1}^\Omega$ to be the class of weights $w$ such that 
\[
[w]_{A_1^\Omega} \coloneqq \sup_{x\in\R^n} \frac{\M_{\Omega}w(x)}{w(x)}<\infty.
\]
A well known class of Muckenhoupt weights is produced be considering $\Omega=\{ {e_1},\ldots,  {e_n}\}$; then $A_p ^\Omega$ is just the class $A_p ^*$ of \emph{strong} or \emph{$n$-parameter} Muckenhoupt weights. We also note that an obvious corollary of one dimensional theory is that 
\[
\|\M_{{\omega}}\|_{L^p(w)\to L^p(w)}\lesssim [w]_{A_p ^\Omega}^\frac{1}{p-1},\qquad {\omega}\in\Omega,
\]
and the implicit constant is independent of $w$ and ${\omega}$. We refer to \cite{Buckley} for the sharp one-dimensional weighted bound for $\M_{{\omega}}$.

%%%%%%%%%%%%%%%%%%%%%%%%%%%%%% SECTION  SECTION SECTION
\subsection{Extrapolation for $A_p ^\Omega$ weights}\label{sec:extrap} Having established the appropriate $A_p ^\Omega$ classes,  we now proceed to proving one of the most useful properties of weighted norm inequalities, that of extrapolation.

We begin by noting that, as in the case of classical $A_p$ weights, it is easy to create $A_p ^\Omega$-weights by using the Rubio de Francia method and factorization; see \cite[Lemmata 2.1,2.2]{Duo2011}. We omit the  proofs which are essentially identical to the one-directional case.

%%%%%%%%%%%%%%%%%%%%%%%%%%%%%% LEMMA LEMMA LEMMA
\begin{lemma} \label{lemma:extr1} Let $w\in A_p^\Omega$. For a nonnegative function $g\in L^p(w)$ we define  
\[
Eg\coloneqq \sum_{k=0}^\infty \frac{\mathrm{M}_\Omega^{(k)} g}{2^{k}\|\mathrm{M}_\Omega \|_{L^p(w)}^k}
\] 
Then $Eg$ satisfies the following properties
\begin{enumerate}
	\item[(i)] $\displaystyle g\leq E_g$.
	\item [(ii)] For every $g\in L^p(w)$ we have $\displaystyle \|Eg\|_{L^p(w)}\leq 2\|g\|_{L^p(w)}$.
	\item [(iii)] If $\|\M_\Omega\|_{L^p(w)\to L^p(w)}<\infty$ then $Eg$ is an $A_1 ^\Omega$ weight with constant \[[Eg]_{A_1^\Omega }\leq 2\|\mathrm{M}_\Omega \|_{L^p(w)\to L^p(w)}.\]
\end{enumerate}
Furthermore for all exponents $1\leq p<\infty, 1<p_0<\infty$  and weights $u,w$ there holds
 \[
 \begin{cases}
 [wu^{p-p_0}]_{A_{p_0}^\Omega} \leq  [w ]_{A_{p}^\Omega}  [u ]_{A_{1}^\Omega}, & p\leq p_0 
 \\ 
 &
 \\ 
  \left[w^{\frac{p_0-1}{p-1}}u^{\frac{p-p_0}{p-1}}\right]_{A_{p_0}^\Omega} \leq  [w ]_{A_{p}^\Omega}^{\frac{p_0-1}{p-1}}  [w ]_{A_{p}^\Omega}^{\frac{p-p_0}{p-1}}, & p>p_0.
 \end{cases}
 \]
\end{lemma}
%%%%%%%%%%%%%%%%%%%%%%%%%%%%%% LEMMA LEMMA LEMMA
We now provide the basic extrapolation result for $A_p ^\Omega$ weights which will be our main tool for passing from $L^2(w)$estimates to  $L^p(w)$ estimates for all $p\in (1,\infty)$. This result and its proof are completely analogous to \cite[Theorem 3.1]{Duo2011}, making use of Lemma \ref{lemma:extr1} as the analogous of \cite[Lemmata 2.1, 2.2]{Duo2011}.

%%%%%%%%%%%%%%%%%%%%%%%%%%%%%% LEMMA LEMMA LEMMA
\begin{lemma}\label{lem:Apextrap} Let $\Omega\subset S^{n-1}$ be a set of directions such that for all $1<p<\infty$ and for all $w\in A_p ^\Omega$ we have the weighted boundedness property $\M_\Omega:L^p(w)\to L^p(w)$. Assume that for some family of pairs of nonnegative functions, $(f, g)$, for some $p_0\in[1,\infty]$ and for all $w\in A_{p_0}$ we have
\[
\Big(\int_{\R^n} g^{p_0}w\Big)^\frac{1}{p_0}\leq C \mathcal P([w]_{A_{p_0} ^\Omega})\Big(\int_{\R^n} f^{p_0} w\Big )^\frac{1}{p_0},
\]
where $\mathcal P:\R_+\to \R_+$ is a an increasing function and $C>0$ does not depend on $w$ or the pairs $(f,g)$. Then for all $1<p<\infty$ we have
\[
\Big(\int_{\R^n} g^{p }w\Big)^\frac{1}{p }\leq C \mathcal K(w)\Big(\int_{\R^n} f^{p } w\Big )^\frac{1}{p },
\]
with
\[
\mathcal K(w) \coloneqq \begin{cases} \mathcal P([w]_{A_p ^\Omega}(2\|\M_\Omega\|_{L^p(w)})^{p_0-p}),&\quad p<p_0,
\\
\mathcal P([w]_{A_p ^\Omega} ^\frac{p_0-1}{p-1}(2\|\M_\Omega\|_{L^{p'}(\sigma)})^\frac{p-p_0}{p-1}), &\quad p>p_0.
\end{cases}
\]
\end{lemma}
%%%%%%%%%%%%%%%%%%%%%%%%%%%%%% LEMMA LEMMA LEMMA

%%%%%%%%%%%%%%%%%%%%%%%%%%%%%% SECTION  SECTION SECTION
\subsection{Weighted inequalities for the lacunary directional maximal operator}
In this subsection, we consider directional maximal operators associated to lacunary sets of order $L$. 

According to the previous discussion, the condition $w\in A_p ^\Omega$ is necessary for the boundedness property $\M_{\Omega}:L^p(w)\to L^p(w)$. In this paragraph we also show the sufficiency of condition $A_p ^\Omega$, thus giving a characterization of the $A_p ^\Omega$ class in terms of $\M_{\Omega}$. 

%%%%%%%%%%%%%%%%%%%%%%%%%%%%%% THEOREM THEOREM THEOREM
\begin{theorem}\label{thm:weightedM} Let $\Omega\subset S^{n-1}$ be a  lacunary set of directions of order $L$, where $L$ is a positive integer, and $w$ be a weight. For every $1<p<\infty$, the following are equivalent.
\begin{enumerate}
		\item [(i)] For all $f\in L^p(w)$ we have $\displaystyle \|\M_{\Omega}f\|_{L^p(w)}\lesssim   \|f\|_{L^p(w)}$, with implicit constant depending on $w$, the dimension, and the lacunarity constants of $\Omega$.
		\item[(ii)] We have that $w\in A_p ^\Omega$.
\end{enumerate}
Furthermore, if $w\in A_p ^\Omega $ then we have the estimate $\|\M_\Omega\|_{L^p(w)\to L^p(w)}\lesssim [w]_{A_p ^\Omega} ^{\delta L}$ for some exponent $\delta=\delta(p,n)>0$ and implicit constant independent of $w$.
\end{theorem}
%%%%%%%%%%%%%%%%%%%%%%%%%%%%%% THEOREM THEOREM THEOREM

%%%%%%%%%%%%%%%%%%%%%%%%%%%%%% REMARK REMARK REMARK
\begin{remark} 
We note here that in dimension $n=2$ and for $L=1$ this theorem was known and contained in \cite[Theorem 4]{DuoMoy}. 
\end{remark}
%%%%%%%%%%%%%%%%%%%%%%%%%%%%%% REMARK REMARK REMARK

%%%%%%%%%%%%%%%%%%%%%%%%%%%%%% SECTION SECTION SECTION
\subsection{Proof of Theorem~\ref{thm:weightedM}} 
 Recall from \S\ref{sec:lacsetup} that a set $\Omega\subset S^{n-1}$ is called lacunary of order $L$, where $L\geq 1$ is a positive integer, if there exists a dissection as in \eqref{lacdiss} such that for each $\sigma\in\Sigma(n)$ and $\ell \in \N$, the sets $\Omega_{\sigma,\ell}$ are lacunary of order $L-1$. As metnioned before, cf. Remark~\ref{rem:2-j}, we assume that $\Omega$ is closed and that the axes $\{ {e_1},\ldots, {e_n}\}$ of the dissection of order $L$  are contained in $ \Omega$.  Then, the inclusion $A_p ^\Omega \subset A_p ^*$ holds, the latter being the class of strong $A_p$ weights with respect to these coordinate axes. In consequence, the strong maximal function $\mathrm M_{\mathsf{s}}$ is automatically bounded on $L^p(w)$ for $w\in A_p^\Omega$.

As in the proof of \cite[Theorem A]{PR1} we rely on the covering of the singularity hyperplane $\xi \cdot {\omega}=0$ 
 by finitely overlapping unions of two dimensional wedges $\{\Psi_{\sigma, \ell_\sigma} :\, \sigma \in \Sigma\}$ defined in \eqref{eq:wedgesp}, where ${\cic{\ell}}=(\ell_\sigma:\,\sigma \in \Sigma)$ is the unique index in $\mathbb Z^\Sigma$ such that  ${\omega}$ belongs to the cell $\Omega_{{\cic{\ell}}}$.   The core of the proof is contained in the following two lemmata which are weighted versions of the corresponding results from \cite{PR1}.

The first result we need is a weighted analogue of \cite[Lemma 1.1]{PR1}. Note that it does not require the lacunarity assumption on $\Omega$ and the weight class needed is just the usual class of strong Muckenhoupt weights $A_p ^*$.
%%%%%%%%%%%%%%%%%%%%%%%%%%%%%% LEMMA LEMMA LEMMA
\begin{lemma}\label{lem:splitw}Let $p>1$ and $w\in A_p ^*$ be a weight. 
There holds
	\[
	\|\M_{\Omega}f \|_{L^p(w) }\lesssim \mathcal [w]_{A_p ^*} ^\frac{n}{p-1} \sup_{\varnothing \neq U\subseteq \Sigma}\big\|\sup_{{\cic{\ell}}\in\mathbb Z^\Sigma } \M_{\Omega_{{\cic{\ell}}}}   K_{U,{\cic{\ell}}}  f \big\|_{L^p(w)},
	\]
with the implicit constant depending upon dimension and $p$.
\end{lemma}
%%%%%%%%%%%%%%%%%%%%%%%%%%%%%% LEMMA LEMMA LEMMA

%%%%%%%%%%%%%%%%%%%%%%%%%%%%%% PROOF PROOF PROOF
\begin{proof} The proof follows from the arguments in the proof of \cite[Lemma 1.1]{PR1}. Indeed one just needs to note that the corresponding unweighted estimate in \cite{PR1} is proved via the use of pointwise estimates, which of course are independent of the underlying measure, and the boundedness of the strong maximal function $\mathrm M_{\mathsf{s}} f$ on $L^p(\R^n)$. The latter fact is replaced by the observation that $\mathrm M_{\mathsf{s}} $ maps $L^p(w)$ to itself whenever $w\in A_p ^*$, and satisfies the quantitative norm estimate
	\[
\|	\mathrm M_{\mathsf{s}} \|_{L^p(w)\to L^p(w)}\lesssim [w]_{A_p ^*} ^\frac{n}{p-1}.
	\]
Here again we use the one-dimensional sharp weighted estimate for the Hardy-Littlewood maximal operator from \cite{Buckley}.
\end{proof}
%%%%%%%%%%%%%%%%%%%%%%%%%%%%%% PROOF PROOF PROOF

%%%%%%%%%%%%%%%%%%%%%%%%%%%%%% LEMMA LEMMA LEMMA
The second result is a weighted  square function estimate for the angular multipliers $K_{U,{\cic{\ell}}}$ associated to a lacunary dissection of the sphere.
\begin{lemma}\label{lem:PRlemmaw}  Let $1<p<\infty$ and $\Sigma$ corresponding to a given dissection of the sphere. Then for all $w\in A_p ^*$ we have
\[ \sup_{U\subseteq \Sigma }\Big\| \Big(\sum_{{\cic{\ell}} \in \mathbb Z^{U}}  \big|   K_{U,{\cic{\ell}}} f\big|^2\Big)^{\frac12} \Big\|_{L^p(w)} \lesssim \mathcal  [w]_{A_p ^*}^\beta \|f\|_{L^p(w)}.
\]
The implicit constant depends upon dimension $n$, $p$, and $\beta>0$ depends on $p$ and $n$.
\end{lemma}
%%%%%%%%%%%%%%%%%%%%%%%%%%%%%% LEMMA LEMMA LEMMA

%%%%%%%%%%%%%%%%%%%%%%%%%%%%%% PROOF PROOF PROOF
\begin{proof} As in the proof of \cite[Lemma 1.2]{PR1} we note that it will be enough to prove the $L^p(w)$-boundedness of the randomized map $T$ given as
	\[  
	f\mapsto  \Big(\sum_{{\cic{\ell}} \in \mathbb Z^{S}}  \eps_{{\cic{\ell}}}  \prod_{\sigma\in U} \kappa^{\circ}_{\sigma,\ell_\sigma}  \hat f \Big)^\vee\eqqcolon (m \hat f)^\vee,
	\]
uniformly over choices of signs $\{\eps_{{\cic{\ell}}}\}_{{\cic{\ell}} \in \mathbb Z^U}$. The unweighted $L^2(\R^n)$-boundedness of this map follows simply by Plancherel and the finite overlap property of the supports $\{\widetilde{\Psi}_{\sigma,\ell}:\, \ell\in\N\}$, which shows that $m\in L^\infty$, uniformly over choices of signs. For $L^p(w)$-bounds, we need an $A_p^*$-weighted version of the standard Marcinkiewicz multiplier theorem. This can be found for example in \cite[Theorem 3]{Kurtz} so the proof of the lemma reduces to checking a number of conditions on averaged derivatives of $m$. In fact these conditions are identical to the hypothesis of the unweighted Marcinkiewicz multiplier theorem, as can be found for example in \cite[p. 109]{stein} and can be verified by using estimates \eqref{eq:fmoest} for each single multiplier $K_{\sigma,\ell}^\circ$.   An inspection of the proof, which relies on the weighted vector valued boundedness of frequency projections on rectangles, and the weighted multiparameter Littlewood-Paley inequalities, shows that there exists a constant $\beta$ depending on $n$ and $p$ such that  $\|T\|_{L^p(w)\to L^p(w)}\lesssim [w]_{A_p ^*} ^\beta$. 
\end{proof}
%%%%%%%%%%%%%%%%%%%%%%%%%%%%%% PROOF PROOF PROOF

We now give the conclusion of the proof of Theorem~\ref{thm:weightedM}.
%%%%%%%%%%%%%%%%%%%%%%%%%%%%%% PROOF PROOF PROOF
\begin{proof}[Conclusion of the proof of Theorem~\ref{thm:weightedM}]   The key step is the estimate
\begin{equation}\label{eq:weightedmain}
	\|\M_\Omega\|_{L^p(w)\to L^p(w)}\lesssim \mathcal   [w]_{A_p ^*} ^\delta \sup_{\sigma\in\Sigma}\sup_{\ell \in \mathbb Z}\|\M_{\Omega_{\sigma,\ell}}\|_{L^p(w)\to L^p(w)},
\end{equation}
for some exponent $\delta>0$ depending on the dimension $n$ and on $p$. Indeed, if $L=1$, each sector $\Omega_{\sigma,\ell}$ contains at most one direction, whence using the well-known  weighted maximal inequality for each such direction
and the obvious inequality  $[w]_{A_p^{\{{\omega}\}}}\leq [w]_{A_p ^\Omega}$ for ${\omega}\in \Omega$, 
\[
\sup_{\sigma\in\Sigma}\sup_{\ell \in \mathbb Z}\|\M_{\Omega_{\sigma,\ell}}\|_{L^p(w)\to L^p(w)}
 \lesssim  \sup_{{\omega}\in \Omega}\|\M_{\{{\omega}\}}\|_{L^p(w)\to L^p(w)}.
\]
Coupling the latter display with \eqref{eq:weightedmain} yields the claimed estimate in Theorem~\ref{thm:weightedM}. 
We now proceed by induction and derive the  $L$-lacunary case assuming the $L-1$ holds true. Estimate \eqref{eq:weightedmain} and the inductive assumption read
\[
	\|\M_\Omega\|_{L^p(w)\to L^p(w)}\lesssim \mathcal   [w]_{A_p ^*} ^\delta \sup_{\sigma\in\Sigma}\sup_{\ell \in \mathbb Z}\|\M_{\Omega_{\sigma,\ell}}\|_{L^p(w)\to L^p(w)} \lesssim  [w]_{A_p ^*} ^\delta \sup_{\sigma\in\Sigma}\sup_{\ell \in \mathbb Z} [w]_{A_p^{\Omega_{\sigma,\ell}}}^{\delta(L-1)} \lesssim [w]_{A_p^{\Omega}}^{\delta L} 
\]
where in the last inequality we have used the obvious fact that  $\sup_{\sigma,\ell}[w]_{A_p^{\Omega_{\sigma,\ell}}}\leq [w]_{A_p^{\Omega}}$ and $[w]_{A_p^*}\leq [w]_{A_p^{\Omega}}$. This completes the proof of the theorem up to showing estimate \eqref{eq:weightedmain} holds true. \end{proof}
\begin{proof}[Proof of \eqref{eq:weightedmain}] We first perform the proof in the case $p\geq 2$. Let us for a moment fix a $U\subseteq \Sigma$ and write $\mathbb{Z}^\Sigma = \mathbb{Z}^U \otimes \mathbb{Z}^{\Sigma\setminus U}$ so that given ${\cic{\ell}}=\{\ell_\sigma\}_{\sigma\in\Sigma}$ we decompose ${\cic{\ell}} =  \cic\tau \times \cic t$ with $\cic \tau=\{\tau_\sigma:\, \sigma\in U\}$. Replacing the supremum by an $\ell^p$ function gives
	\begin{equation}\label{eq:(6)}
	\Big\|\sup_{{\cic{\ell}}\in \mathbb{Z}^\Sigma}\M_{\Omega_{{\cic{\ell}}}} f_{\cic \tau}\Big\|_{L^p(w)}\leq  \sup_{\sigma\in\Sigma} \sup_{\ell \in\mathbb{Z}}\|\M_{\Omega_{\sigma,\ell}}\|_{L^p(w)\to L^p(w)} \Big\|\big(\sum_{\cic\tau\in\mathbb{Z}^U}|f_{\cic \tau}|^p\big)^\frac{1}{p}\Big\|_{L^2(w)}.
	\end{equation}
As $p\geq 2$ estimate~\eqref{eq:(6)} implies 
\[
\Big\|\sup_{\cic{ \ell}\in \mathbb{Z}^\Sigma}\M_{\Omega_{\cic \ell}} f_{\cic{ \tau}}\Big\|_{L^p(w)}\leq \sup_{\sigma\in\Sigma} \sup_{\ell \in\N}\|\M_{\Omega_{\sigma,\ell}}\|_{L^p(w)\to L^p(w)} \Big\|\big(\sum_{\cic \tau\in\mathbb{Z}^S}|f_{\cic \tau}|^2\big)^\frac{1}{2}\Big\|_{L^p(w)}.
\]
Now \eqref{eq:weightedmain}  follows by taking $f_{\cic \tau}\coloneqq   {K_{U,\cic\tau}} $ where $\cic \tau=\{\tau_\sigma:\, \sigma\in U\}$ and bounding the right hand side in the last display, from above, by Lemma~\ref{lem:PRlemmaw}, and the left hand side of the last display, from below, by Lemma~\ref{lem:splitw}.
For $1<p<2$ we note that, by monotone convergence, it suffices to show the estimate for every finite subset of $\Omega$, which we still call $\Omega$. Then $\|\M_\Omega\|_{L^p(w)\to L^p(w)}<\infty$, $w\in A_p ^\Omega$, so we can interpolate between the estimates
\[
\Big\|\sup_{\cic \ell\in {\mathbb Z}^\Sigma}\M_{\Omega_{\cic \ell}} f_{\cic \tau}\Big\|_{L^p(w)}\leq\|\M_{\Omega_{\cic \ell}} \|_{L^p(w)\to L^p(w)} \Big\|\sup_{\cic \ell\in {\mathbb Z}^\Sigma}|f_{\cic \tau}|\Big\|_{L^p(w)}
\]
and \eqref{eq:(6)} to conclude
\[\begin{split}&\quad 
\Big\|\sup_{\cic \ell\in {\mathbb Z}^\Sigma}\M_{\Omega_{\cic \ell}} f_{\cic \tau}\Big\|_{L^p(w)}\\ &\leq \|\M_\Omega\|_{L^p(w)\to L^p(w)}^{1-\frac{p}{2}}\Big(\sup_{\sigma\in\Sigma}\sup_{\ell\in{\mathbb Z}}\|\M_{\Omega_{\sigma,\ell}}\|_{L^p(w)\to L^p(w)}\Big)^\frac{p}{2}\Big\| \Big(\sum_{\cic\tau\in{\mathbb Z}^S}|f_{\cic\tau}|^2\Big)^\frac{1}{2}\Big\|_{L^p(w)}.
\end{split}
\]
Taking again $f_{\cic \tau}={K_{U,\cic\tau}} $  an application of Lemmata~\ref{lem:PRlemmaw} and \ref{lem:splitw} yields
\[
\|\M_\Omega\|_{L^p(w)\to L^p(w)}\lesssim [w]_{A_p ^*} ^\gamma \|\M_{\Omega}\|_{L^p(w)\to L^p(w)} ^{1-\frac{p}{2}}  \Big(\sup_{\sigma\in\Sigma}\sup_{\ell\in{\mathbb Z}}\|\M_{\Omega_{\sigma,\ell}}\|_{L^p(w)\to L^p(w)}\Big)^\frac{p}{2}
\]
with $\gamma=\beta+n/(p-1)$. As we have assumed that $\|\M_\Omega\|_{L^p(w)\to L^p(w)}<\infty$ we may rearrange and complete the proof of the theorem.
\end{proof}

%%%%%%%%%%%%%%%%%%%%%%%%%%%%%% PROOF PROOF PROOF

%%%%%%%%%%%%%%%%%%%%%%%%%%%%%% SECTION SECTION SECTION
\section{An almost orthogonality principle for the maximal Hilbert transform} \label{sec:aop} We now prove an almost orthogonality principle for the maximal Hilbert transform of a set $\Omega\subset S^{2}$. 
In the statements below it is convenient to  write for all nonnegative integers $N$, weights $w$ on $\R^{3}$, and $\Omega\subset S^{2}$
\[
\Theta_N(\Omega,w)\coloneqq\sup_{\substack{O \subset \Omega\\ \# O \leq N} }\left\| H_O\right\|_{L^2(\R^3;w)\to L^2(\R^3;w)}.
\]

%%%%%%%%%%%%%%%%%%%%%%%%%%%%%% THEOREM THEOREM THEOREM
\begin{theorem}\label{th:main} There exist $C,\gamma\geq 1$  such that the following holds. Let    $N$ be a positive integer,  $\mathcal B$ be a choice of ONB,  $\Omega\subset S^{2}$ a set of directions containing $\mathcal B$ and $w\in A_2^\Omega$. Then 
\[
\Theta_N(\Omega,w) \leq C  [w]_{A_2^\Omega}^\gamma \left[\sqrt{\log N}   + \sup_{\sigma \in \Sigma } \sup_{\ell \in \mathbb Z} \Theta_N(\Omega_{\sigma,\ell},w)\right]
\]
where the lacunary dissection is taken with respect to $\mathcal B$ as in \eqref{lacdiss}.
\end{theorem}
%%%%%%%%%%%%%%%%%%%%%%%%%%%%%% THEOREM THEOREM THEOREM
By iterative application of the almost orthogonality principle, and extrapolation, we obtain the following corollary, of which Theorem \ref{th:mainintro} is the particular case $w=1$.

%%%%%%%%%%%%%%%%%%%%%%%%%%%%%% COROLLARY COROLLARY COROLLARY
\begin{corollary}\label{cor:main} Let $1<p<\infty$ and  $L\geq 0$. There exists constants $C=C_{p,L},\gamma=\gamma_{p,L}$ such that for any  $\Omega\subset S^{2}$    lacunary set of order $L$  and $w\in A_p^\Omega$   
\[
\sup_{\substack{O\subset \Omega \\ \# O\leq N}} \left\|H_{O}\right\|_{L^p(\R^3; w)\to L^p(\R^3; w)} \leq C [w]_{A_p^\Omega}^\gamma \sqrt{\log N}.
\]
\end{corollary}
%%%%%%%%%%%%%%%%%%%%%%%%%%%%%% COROLLARY COROLLARY COROLLARY

The proof of Theorem \ref{th:main} rests upon the results of the previous sections, as well as on the proposition below,  a weighted version of the Chang-Wilson-Wolff principle, which we state and prove before the main argument. In the statement of the proposition below we remember that $A_p ^*=A_p ^{\Omega}$ with $\Omega$ being the canonical basis of $\R^n$, namely $\Omega=\{ {e_1},\ldots,  {e_n}\}$. We also use the standard notation $A_\infty ^*\coloneqq \cup_{p>1} A_p ^*$.

%%%%%%%%%%%%%%%%%%%%%%%%%%%%%% PROPOSITION PROPOSITION PROPOSITION
\begin{proposition}\label{prop:weightedcww}
 Let  $\{K_1,\ldots,K_N\}$ be Fourier multiplier operators on $\R^n$ with uniform bound \[
\sup_{1\leq j\leq N}\|K_j\|_{L^2(w)\to L^2(w)} \leq   [w]_{A_2^*}^\alpha
\]
for some $\alpha>0$.
Let $\{P_t ^\upsilon\}_{t\in\mathbb Z}$ be a smooth Littlewood-Paley decomposition acting on the $\upsilon$-th frequency variable, where $1\leq \upsilon\leq n$. For $w\in A_p ^*$ and $1<p<\infty$ we then have
	\[
	\Big\|\sup_{1\leq j\leq N} |K_j f|\Big\|_{L^p(w)}\lesssim  [w]_{A_p^\star}^\gamma  
	\bigg[ \|f\|_{L^p(w)} +   (\log (N+1))^\frac{1}{2} \Big\| \big(\sum_{t\in\mathbb Z}  \sup_{1\leq j\leq N} |K_j P_t^\upsilon f|^2 \big)^\frac{1}{2} \Big\|_{L^p(w)}\bigg]
	\]
for some exponent $\gamma=\gamma(\alpha,p,n)$ and implicit constant depending on   $\alpha,p,n$.
\end{proposition}
%%%%%%%%%%%%%%%%%%%%%%%%%%%%%% PROPOSITION PROPOSITION PROPOSITION

%%%%%%%%%%%%%%%%%%%%%%%%%%%%%% PROOF PROOF PROOF
\begin{proof}  To simplify the notation we work with  $\upsilon=1$ and set
\[
K^\star f\coloneqq\sup_{1\leq j\leq N} |K_j f|, \qquad D K^\star\coloneqq \Big(\sum_{t\in\mathbb Z}  \sup_{1\leq j\leq N} |K_j P_t^1 f|^2 \Big)^\frac{1}{2}.
\]
Let $\{\mathcal D_j:\,j \in \mathbb Z\}$ be the standard dyadic filtration on $\R$, $\mathbb E_j$ be the associated sequence of conditional expectations, and $\Delta f$ denote the associated martingale square function. Let $\mathbb E_j^1$ be the sequence of conditional expectations on $L^1(\R^n)$ acting on tensor products $f(x)=g(x_1)\otimes h(x_2,\ldots,x_n)$ by  $\mathbb E_j^1 f \coloneqq \mathbb E_j g \otimes h$ and denote by $\Delta^1 f$ the associated martingale square functions.
The Chang-Wilson-Wolff inequality \cite{CWW} tells us that if $w$ is an $A_\infty ^*$-weight then
\begin{equation}
\label{eq:CWW}
\begin{split}
&\quad 
w\left(\left\{x\in \R^n:\, |g(x) - \mathbb E_0^1 g(x)|>2\lambda, |\Delta^1 f(x)|\leq \gamma \lambda \right\}\right)
\\ 
&\qquad  \leq A \exp\bigg( -\frac{b}{[w]_{A_\infty}^{(1)} \gamma^2}\bigg)w\left(\left\{x\in \R^n: \, | { \mathrm{M}_{e_1}} g(x)|> \lambda   \right\}\right),
\end{split}
\end{equation}
where $A,b$ are absolute positive constants  and 
\[
[w]_{A_\infty^{(1)}}\coloneqq \sup_{x\in \R^n} [w(x+\cdot e_1 )]_{A_\infty};
\]
here $[\cdot]_{A_\infty}$ denotes the Wilson $A_\infty$ constant of a weight on the real line, see \cite{Wilson}. The inequality \eqref{eq:CWW} for $n>1$ is in fact obtained  from the one dimensional version of \cite{CWW} and Fubini.  As $[w]_{A_\infty^{(1)}}\leq [w]_{A_p^*}$ and proceeding exactly as in 
  in the proof of  \cite[Corollary 1.14]{DPGTZK}  we can use the above inequality to reach 
\begin{equation}
\label{eq:CWWpf}
\begin{split} &\quad 
\|K^\star f\|_{L^{p}(w)}  
\\ 
&\lesssim \| \mathrm{M}_{e_1}f\|_{L^{p}(w)} +\sup_{1\leq j\leq N}\| \mathrm{M}_{e_1} K_{j} f\|_{L^{p}(w)}  + \sqrt{\log (N+1)}\| \widetilde{ \mathrm{M}}_{e_1}  \|_{L^{p}(w)}  \|  D K^\star f\|_{L^{p}(w)} 
\end{split}
\end{equation}
where $ \widetilde{ \mathrm{M}}_{e_1}f\coloneqq(\mathrm{M}_{e_1}|f|^r)^\frac{1}{r}$, and  $r>1$ can be chosen arbitrarily close to $1$; the implicit constant depends on $p,r$, and polynomially on  $[w]_{A_p^*}$. Since our weight $w\in A_p ^*$, ${ \mathrm{M}}_{e_1}$ is a bounded operator on $L^p(w)$. Furthermore, using the reverse H\"older property for $A_\infty^*$ weights, see e.g.\ \cite[Theorem 1.4]{HP}, we actually have the openness property
\[
[w]_{A_{\frac pr}^*} \leq 2 [w]_{A_{p}^*}, \qquad r \leq \frac{p}{p-c([w]_{A_\infty^*})^{-1}},
\]
where the positive constant $c=c(p,n)\leq 1$  can be explicitly computed. 
Therefore $\widetilde{ \mathrm{M}}_{e_1}$ is also a bounded operator on $L^p(w)$ provided $r$ is chosen small enough to comply with the restriction in the last display. Making use of these $L^p(w)$-bounds in \eqref{eq:CWWpf}  finally yields  the proposition.
\end{proof}
%%%%%%%%%%%%%%%%%%%%%%%%%%%%%% PROOF PROOF PROOF

%%%%%%%%%%%%%%%%%%%%%%%%%%%%%% PROOF PROOF PROOF
\subsection{Proof of Theorem \ref{th:main}}  In this proof the implicit constants occurring in the inequalities as well as the exponent $\gamma$ are meant to be absolute and are allowed to vary without explicit mention. Let $\Omega\subset S^2$ and an ONB $\mathcal B\subset \Omega$ be given. Fix a subset $O\subset \Omega$ with $\#O=N$. Of course the set of addresses of the cells whose intersection with $O$ is nonempty, in symbols
$
\mathbb L_O \coloneqq \{{\cic{\ell}} \in \mathbb Z^\Sigma:\,  O\cap S_{{\cic{\ell}}}\neq \emptyset\}
$,
has cardinality at most $N$.
 We use the pointwise estimate of   Lemma \ref{lemma:represent} for each ${\omega}\in O $ to obtain that 
\begin{equation}
\label{eq:pw1}
\begin{split}
 |H_{O} f|  &\lesssim |f| +  \sup_{\varnothing \subsetneq U\subseteq \Sigma}\sup_{\cic{\ell} \in {\mathbb L_O}  } \left|H_{O\cap S_{\cic{\ell}}}   K_{U,\cic{\ell}}f \right|   + \sup_{\varnothing \subsetneq U\subseteq \Sigma} \sup_{{\cic{\eps}} \in \{+,-\}^U}   \sup_{{\cic{\ell}} \in {\mathbb L_O}  } \big|   K_{U,{\cic{\ell}}}^{{\cic{\eps}} } f\big|.
\end{split}
\end{equation}
We may ignore the first summand on the right hand side. We bound the norm of the second summand on the right hand side by a constant multiple of 
\begin{equation}
\label{eq:pw2}
\begin{split}
&\quad \sup_{\varnothing \subsetneq U\subseteq \Sigma} \bigg\|\Big( \sum_{{\cic{\ell}} \in {\mathbb L_O}  } \big|H_{O\cap S_{{\cic{\ell}}}}   K_{U,{\cic{\ell}}}f \big|^2\Big)^{\frac12} \bigg\|_{L^2(w)} \leq B  \sup_{\varnothing \subsetneq U\subseteq \Sigma} \bigg\|\Big( \sum_{{\cic{\ell}} \in \mathbb Z^U } \big|   K_{U,{\cic{\ell}}}f \big|^2\Big)^{\frac12} \bigg\|_{L^2(w)} 
\\ 
& \qquad \lesssim [w]_{A_2^\Omega}^\gamma\Big( \sup_{\sigma \in \Sigma } \sup_{\ell \in \mathbb Z} \Theta_N(\Omega_{\sigma,\ell},w) \Big) \|f\|_{L^2(w)},
\end{split}
\end{equation}
where in the last step we have used the weighted estimate of Lemma \ref{lem:PRlemmaw}, and we have also used the easy estimate
\[
B\coloneqq \sup_{ \|g\|_{L^{2}(w;\ell^2)}=1}  \bigg\|\Big(\sum_{{\cic{\ell}} \in \mathbb Z^\Sigma}\big|H_{O\cap S_{{\cic{\ell}}}}  g_{{\cic{\ell}}}\big|^2\Big)^{\frac12} \bigg\|_{L^2(w)} \leq \sup_{\sigma \in \Sigma } \sup_{\ell \in \mathbb Z} \Theta_N(\Omega_{\sigma,\ell},w).
\]
The third summand in \eqref{eq:pw1} is treated in the next Proposition. In fact,  coupling the bounds \eqref{eq:pw2} above, and  \eqref{eq:pw3} below, with the pointwise estimate \eqref{eq:pw1}, and noticing that $[w]_{A_2^*}\leq[w]_{A_2^\Omega}$ since the coordinate basis vectors are contained in $\Omega$, completes the proof of Theorem \ref{th:main}.
%%%%%%%%%%%%%%%%%%%%%%%%%%%%%% PROOF PROOF PROOF

%%%%%%%%%%%%%%%%%%%%%%%%%%%%%% PROPOSITION PROPOSITION PROPOSITION
\begin{proposition}\label{prop:smoothm} Let $\mathbb L$ be a finite  subset of $\,\mathbb {Z}^3$. Then 
\begin{equation}\label{eq:pw3}
\sup_{\varnothing \subsetneq U\subseteq \Sigma} \sup_{{\cic{\eps}} \in \{+,-\}^U}  \Big\|  \sup_{{\cic{\ell}} \in {\mathbb L} }\big|   K_{U,{\cic{\ell}}}^{{\cic{\eps}} } f\big|\Big\|_{L^2(w)} \lesssim [w]_{A_2^*}^\gamma \sqrt{\log (\# {\mathbb L}+1)} \|f\|_{L^2(w)}.
\end{equation}
\end{proposition}
%%%%%%%%%%%%%%%%%%%%%%%%%%%%%% PROPOSITION PROPOSITION PROPOSITION

%%%%%%%%%%%%%%%%%%%%%%%%%%%%%% PROOF PROOF PROOF
\begin{proof} Fix $U\subseteq \Sigma, {\cic{\eps}} \in \{+,-\}^U $ throughout the proof.  By means of compositions of Hilbert transforms along the coordinate directions we may decompose
\[
f= \sum_{Q} f_Q, \qquad \|f_Q\|_{L^2(w)} \lesssim   [w]_{A_2 ^*} ^ 3 \|f\|_{L^2(w)},
\]
where each $f_Q$ has frequency support in one of the octants $Q$ of $\R^3$.  By virtue of the norm estimate of the above display, we may fix one of these octants $Q$ and prove \eqref{eq:pw3} for functions $f$ whose frequency support is contained in $Q$, which we do here onwards. Now we remember that by Lemma~\ref{lem:PRlemmaw} the multiplier operators $\{K_{U,{\cic{\ell}}}^{{\cic{\eps}} }:\, {\cic{\ell}} \in   {\mathbb{L}}\}$ satisfy weighted $L^2$ bounds with weighted operator norms bounded polynomially in $[w]_{A_2 ^*}$, uniformly in $U$ and $\cic{\ell}$.
This allows us to use Proposition \ref{prop:weightedcww} on the $  N= \# \mathbb{L}\,$  Fourier multiplier operators $\{K_{U,{\cic{\ell}}}^{{\cic{\eps}} }:\, {\cic{\ell}} \in   {\mathbb{L}}\}$, to get that
\begin{equation}
\label{eq:pw4}
\begin{split}
&\quad  \Big\|  \sup_{{\cic{\ell}} \in {\mathbb L }  }\big|   K_{U,{\cic{\ell}}}^{{\cic{\eps}} } f\big|\Big\|_{L^2(w)} \\ &\lesssim [w]_{A_2^*} ^\gamma\bigg[\|f\|_{L^2(w)}+ \sqrt{\log (N+1)}  \Big\| \Big(\sum_{t\in\mathbb Z}  \sup_{{\cic{\ell}} \in {\mathbb L }  } |K_{U,{\cic{\ell}}}^{{\cic{\eps}} } P_t^\upsilon f|^2 \Big)^\frac{1}{2} \Big\|_{L^2(w)}\bigg]
 \end{split}
\end{equation}
 for any $\upsilon\in \{1,2,3\}$. We  make the choice $\upsilon=\upsilon(U,{\cic{\eps}},Q)\in \{1,2,3\}$ according to Lemma \ref{lemma:annular}, so that based on $\mathrm{supp}\, \widehat f\subset Q$
\[
 \big|K_{U,{\cic{\ell}}}^{{\cic{\eps}} }(P_{t}^{\upsilon} f)(x)\big| \lesssim \mathrm M_{\mathsf{s}}^2 (P_{t}^{\upsilon} f) (x).
\]
Combining the last two inequalities followed by weighted Fefferman-Stein and Littlewood-Paley estimates
\[
\begin{split}
\Big\|  \sup_{{\cic{\ell}} \in {\mathbb L }  }\big|   K_{U,{\cic{\ell}}}^{{\cic{\eps}} } f\big|\Big\|_{L^2(w)} &\lesssim[w]_{A_2^*} ^\gamma\bigg[\|f\|_{L^2(w)}+\sqrt{\log (N+1)}   \Big\| \Big(\sum_{t\in\mathbb Z}  \mathrm M_{\mathsf{s}}^2 (P_{t}^{\upsilon} f)^2 \Big)^\frac{1}{2} \Big\|_{L^2(w)} \bigg]
 \\ 
 & \lesssim  [w]_{A_2 ^*}^\gamma \sqrt{\log (N+1)} \|f\|_{L^2(w)}
\end{split}
\]
which is the claimed \eqref{eq:pw3}. 
\end{proof}
%%%%%%%%%%%%%%%%%%%%%%%%%%%%%% PROOF PROOF PROOF

%%%%%%%%%%%%%%%%%%%%%%%%%%%%%% SECTION SECTION SECTION
\section{Quantitative counterexamples for the model operator} \label{sec:cex}
In this section, we show that sharp higher dimensional $(n\geq 4)$ analogues of  Theorem \ref{th:mainintro} cannot be attacked by means of the   model operators of Section \ref{sec:model}, which are essentially compositions of smooth two-dimensional lacunary cutoffs. To wit, we show that the maximal operators 
\[
\sup_{\cic{\ell} \in \mathbb L }  \Big|\Big[ \prod_{\sigma \in \Sigma}  \big(\mathrm{Id} - K_{{\sigma},  \ell _\sigma}^{\eps_\sigma}\big)\Big] f\Big|, \qquad \sup_{\cic{\ell} \in \mathbb L    }\big|   K_{U,{\cic{\ell}}}^{{\cic{\eps}} } f\big|,
\]
intervening in the decomposition of the maximal Hilbert transform induced by Lemma \ref{lemma:represent}, have operator norms which grow at order $( \log \#\mathbb  L  )^{\frac12\lfloor\frac n2 \rfloor}$. For $n\geq 4$, this is   unfavorable  compared to the maximal Hilbert transform over finite subsets $O$ of a (finite order) lacunary set $\Omega$, whose operator norm is of order at most $\log(\#O)$; see \cite[Corollary 4.1]{PR2}. Our counterexamples are obtained by careful  tensoring of the lower bound for the two-dimensional case $\Sigma=\{(1,2)\}$
which in turn descends from the main theorem of \cite{Karag}.

We use the notation of Section \ref{sec:model} and in particular of \eqref{eq:fmo}. However in this section it will be more convenient to use the equivalent (up to identity) definition
\[
H_{{\omega}} f(x) \coloneqq \int_{\R^n} \widehat{f}(\xi) \cic{1}_{(0,\infty)}  (\xi \cdot {\omega}) \e^{ix\cdot \xi }\, \d \xi .
\] 
%%%%%%%%%%%%%%%%%%%%%%%%%%%%%% SECTION SECTION SECTION
\subsection{A lower bound in $n=2$} 
 The lower bound for $p=2$ of Karagulyan \cite{Karag} combined with the upper bound for all $1<p<\infty $ of \cite{DDP,DPP}
tells us that for all $L\geq 0$ and $1<p<\infty$ there exists $c_{p,L}>0$ such that the following holds: Whenever $\Omega\subset S^1$ is a lacunary set of order $L$ and $O\subset \Omega$ is finite there exists a Schwartz function $f_O$ with
\begin{equation}
\label{eq:lbH}
\|f_O\|_{L^p(\R^2)}=1, \qquad 
\|H_O f_O\|_{L^p(\R^2)}\geq c_{p, L} \sqrt{\log \#O}.
\end{equation}
Let now $\Omega$ be a lacunary set of order 1 with   $\Omega\subset \{{\omega}\in S^1:\,\omega_1,  \omega_2>0\}$. We can take $f_O$ to be frequency supported in the quadrants $\{\xi \in \R^2:\,\xi_1\xi_2<0\}$ as $H_O$ acts trivially on the remaining frequency plane. By a symmetry argument we can actually take 
\begin{equation}\label{eq:suppfO}
	 \mathrm{supp}\,\widehat{f_O}\subset Q_{(1,2)}\coloneqq \big\{\xi \in \R^2:\, \xi_1>0, \, \xi_2<0\big\}.
\end{equation}
Rewriting \eqref{eq:split1} in this particular case we see that if ${\omega} \in \Omega \cap S_{(1,2),\ell(\omega)}$ and $\mathrm{supp}\,\widehat{f }\subset Q_{(1,2)}$ then
\[
\begin{split}
\widehat{H_{{\omega}} f}(\xi) & = \mathcal{F}\big(H_{{\omega}}K_{(1,2),\ell(\omega)}^\circ f\big)(\xi) + \cic{1}_{(0,\infty)}  (\xi \cdot {\omega}) \big(1- \kappa_{(1,2),\ell(\omega)}^\circ (\xi)\big) \widehat{f}(\xi) 
\\ 
& = \mathcal{F}\big(H_{{\omega}}K_{(1,2),\ell(\omega)}^\circ f\big)(\xi) +  \big(1- \kappa_{(1,2),\ell(\omega)}^+ (\xi)\big) \widehat{f}(\xi).
\end{split}
\]
We notice that, for some absolute constant $C_p$ 
\[
\Big\|\Big( \sum_{\omega \in \Omega }|H_{{\omega}}K_{(1,2),\ell(\omega)}^\circ f_O|^2\Big)^{\frac12}\Big\|_{L^p(\R^2)}\leq C_p \|f_O\|_{L^p(\R^2)} = C_p;
\]
this $L^p$-boundedness is most easily seen by proving the weighted $L^2$-bound as in Section 5 first. Comparing this last display with \eqref{eq:lbH} we obtain that
\[
\Big\|\sup_{\omega\in O}\big| K_{(1,2),\ell(\omega)}^{+} f_O\big|\Big\|_{L^p(\R^2)} \geq c_p \sqrt{\log \#O}
\]
provided $\#O$ is large enough, with  $c_p=c_{p,1}/2$. The arguments of Section~\ref{sec:model} and symmetry considerations finally show that there exist positive absolute constants $c_p,C_p$ such that for $\eps\in\{+,\circ,-\} $ and all finite index sets  $\mathbb L\subset \mathbb Z$ we have
\begin{equation}
\label{eq:65}
c_p \leq \frac{1}{\sqrt{\log \#\mathbb L}} \Big\| f \mapsto \sup_{\ell \in \mathbb L}\big| K_{(1,2),\ell }^{\eps} f \big|\Big\|_{L^p(\R^2)} \leq C_p.
\end{equation}
\begin{remark}\label{rmk:cex} Just like the maximal Hilbert transform, the maximal operators defined in \eqref{eq:65} are invariant under dilation and reflection through the frequency origin,  and act trivially on functions supported outside $\pm Q_{(1,2)}$. For any fixed $\mathbb L\subset \mathbb Z$ with $\#{\mathbb L}=N$, using the  lower bound in \eqref{eq:65}, the reflection symmetry and an approximation argument we may find $M >0$ and a Schwartz function $f_{\mathbb L}$ with
\begin{equation}
\label{eq:lbweuse}
\mathrm{supp}\,
\widehat{f_{\mathbb L} }\subset Q_{(1,2)}\cap A^{(1,2)}\big(2^{-\frac M 2},2^{\frac M2}\big), \qquad
\big\|  \sup_{\ell \in {\mathbb L}} | K_{(1,2),\ell }^{\eps}f_{{\mathbb L}} |\big\|_{L^p(\R^2)} \geq  C_p\sqrt{\log N}\|f_{{\mathbb L}}\|_{L^p(\R^2)},
\end{equation}
where 
\[
 A^{(1,2)}{(a,b)}\coloneqq \Big\{(\xi_1,\xi_2)\in \R^2: \, a< \sqrt{\xi_1^2+ \xi_2^2} < b\Big\}.
\]
Given any $s\in \R$, the dilation invariance can then be used to find $f_{s,{\mathbb L}}$ with the same properties as $f_{{\mathbb L}}$ in \eqref{eq:lbweuse} but $ \mathrm{supp}\,
\widehat{f_{s,{\mathbb L}} }\subset Q_{(1,2)}\cap A^{(1,2)}(2^{s},2^{s+M})$ .
\end{remark}
The next result is the anticipated counterexample to estimate \eqref{eq:pw3} in dimensions $4$ and higher.
%%%%%%%%%%%%%%%%%%%%%%%%%%%%%% THEOREM THEOREM THEOREM
\begin{theorem} \label{thm:cex}Let $n\geq 2$ be the dimension of the ambient space. Then
 \begin{equation}
\label{eq:pw3_below}
\inf_{{\cic{\eps}} \in \{+,-\}^\Sigma} \sup_{\substack{\, {{\mathbb L}}\subset \mathbb Z^{\Sigma}\\ \# {{\mathbb L}} = N} }\, \sup_{\varnothing \subsetneq U\subseteq \Sigma}  
 \Big\| f\mapsto  \sup_{{\cic{\ell}} \in {\mathbb L} }\big|   K_{U,{\cic{\ell}}}^{{\cic{\eps}} } f\big|\Big\|_{L^p(\R^n)\to L^p(\R^n)} \geq c_p \Big(\sqrt{\log N}\Big)^{ \left\lfloor \frac n 2\right\rfloor}  .
\end{equation}  
\end{theorem}
%%%%%%%%%%%%%%%%%%%%%%%%%%%%%% THEOREM THEOREM THEOREM

%%%%%%%%%%%%%%%%%%%%%%%%%%%%%% PROOF PROOF PROOF
\begin{proof} It suffices to prove the statement for even $n=2d$ and for $N>10 d$, say. By symmetry considerations we may argue in the case where ${\cic{\eps}}=(+,\ldots,+)$.  Let ${\mathbb L}$ be the set of $N^{d}$ indices such that $S_{{\cic{\ell}}} \cap \Omega\neq \emptyset$, where $\Omega$ is the set of vectors on $S^{2d-1}$ obtained by normalizing the vectors $(x_1,\ldots, x_n)$ with components
\[
x_{2k-1}= 2^{-2k N},\quad x_{2k}= 2^{-2k N -m_k},\qquad m_k\in \{1,\ldots,N\},\quad k=1,\ldots,d.
\]
In practice ${\cic{\ell}} ={\cic{\ell}} (m_1,\ldots, m_d)\in {\mathbb L}$ is completely determined by the $2d-1$ conditions
\[
\ell_{(2k-1,2k)}=m_k, \quad k=1,\ldots,d; \qquad \ell_{(2k-1,2k+1)}= 2N, \quad k=1,\ldots,d-1.
\]
As
\[
 1-\prod_{\sigma \in \Sigma}\left(1-\kappa_{\sigma,\ell _\sigma}^{+}\right) = \sum_{\varnothing\subsetneq U \subseteq \Sigma} (-1)^{\# U+1} \prod_{\sigma\in U} \kappa_{\sigma,\ell _\sigma}^{+}
 \]
 estimate \eqref{eq:pw3_below} will follow if we prove that
\begin{equation}
\label{eq:intermediate}
\bigg\| f\mapsto  \sup_{{\cic{\ell}} \in {\mathbb L} }\Big|\Big[ \prod_{\sigma \in \Sigma}  \big(\mathrm{Id} - K_{{\sigma}, \ell _\sigma}^{+}\big)\Big] f\Big|\bigg\|_{L^p(\R^n)\to L^p(\R^n)} \geq c_p \left(\sqrt{\log N}\right)^d,
\end{equation} 
where product denotes composition. Now for each $k=1,\ldots,d$ define the function of two variables  $f_{k}=f_k(x_{2k-1},x_{2k})$ given by $f_{s,{\mathbb L}}$ in Remark \ref{rmk:cex} with the pair $(2k-1,2k) $ in the place of $(1,2)$, with ${\mathbb L}=\{1,\ldots,N\}$, and with $s$ chosen so that 
\[
\mathrm{supp}\, \widehat{f_k} \subset Q_{(2k-1,2k)}\cap A^{(2k-1,2k)}(2^{-3k M},2^{-(3k-1)M}) .
\]
Here 
\[
Q_{(2k-1,2k)}\coloneqq \big\{\xi \in \R^2: \, \xi_{2k-1}>0, \, \xi_{2k}<0\big\}.
\]
We now define
\[
f(x)\coloneqq  \prod_{k=1}^d f_k(x_{2k-1},x_{2k}).
\]
The point of this choice is that if $\sigma=(\sigma(1),\sigma(2))$ is such that $\sigma(1),\sigma(2)$ have the same parity then $\xi_{\sigma (1)},\xi_{\sigma (2)}$ have the same sign on the frequency support of $f$, so that $\big(\mathrm{Id} - K_{{\sigma},  \ell_\sigma}^{+}\big) f=f$. Also, unless  $\sigma=(2k-1,2k)$ for some $k=1,\ldots,d$,  there holds
\[
\begin{split}
&\mathrm{supp}\,\widehat{f} \subset\left\{\xi \in \R^n: \, \frac{|\xi_{\sigma(1)}|}{|\xi_{\sigma(2)}|} \geq  2^{3M}\right\},
\\
&\mathrm{Id} - K_{{\sigma},  \ell _\sigma}^{+} = \mathrm{Id} \quad   \textrm{on the cone} \quad |\xi_{\sigma(1)}|> (2d+1) 2^{- \ell_\sigma}|\xi_{\sigma(2)}|,
\end{split}
\]
which is a larger cone than the one where $\widehat{f}$ is supported, as $ \ell  _\sigma\geq N$ in this case. Summarizing we may delete from the composition in \eqref{eq:intermediate} all the $\sigma$ which are not of the form $\sigma=(2k-1,2k)$, and we have for all $m=1,\ldots,N$ that
\[
   \prod_{\sigma \in \Sigma}  \left(\mathrm{Id} - K_{{\sigma},({\cic{\ell}}(m_1,\ldots,m_d))_\sigma}^{+}\right)  f= \prod_{k=1}^d \left(\mathrm{Id} - K_{{(2k-1,2k)},m_k}^{+}\right)f_k,
\]
with the caveat that the product sign on the left hand side denotes composition while the product sign on the right hand side denotes pointwise product. Therefore using \eqref{eq:lbweuse} for the lower bound in the third line
\[
\begin{split}
&\quad 
\bigg\|  \sup_{{\cic{\ell}} \in \vec{\cic{ \ell}} }\Big|\Big[ \prod_{\sigma \in \Sigma}  \big(\mathrm{Id} - K_{{\sigma}, \ell _\sigma}^{+}\big)\Big] f\Big|\bigg\|_{L^p(\R^n)} \\ & =\bigg\|  \sup_{(m_1,\ldots,m_d)\in \{1,\ldots,N\}^d} \Big|\prod_{k=1}^d \big(\mathrm{Id} - K_{{(2k-1,2k)},m_k}^{+}\big)f_k\Big|\bigg\|_{L^p(\R^n)} 
\\
& =
\bigg\| \prod_{k=1}^d  \sup_{m_k \in  \{1,\ldots,N\}} \left|\left(\mathrm{Id} - K_{{(2k-1,2k)},m_k}^{+}\right)f_k\right|\bigg\|_{L^p(\R^n)} 
\\ 
&   = \prod_{k=1}^d \bigg\| \sup_{m_k \in  \{1,\ldots,N\}} \left|\left(\mathrm{Id} - K_{{(2k-1,2k)},m_k}^{+}\right)f_k\right|\bigg\|_{L^p(x_{2k-1},x_{2k})}
\\
& \geq c_p^d  (\log N)^{\frac d2} \prod_{k=1}^d \|f_k\|_{L^p(x_{2k-1},x_{2k})}
  = c_p^d  (\log N)^{\frac d2}  \|f \|_{L^p(\R^n)}.
\end{split}
\] 
This proves \eqref{eq:intermediate} and thus completes the proof of the theorem. 
\end{proof}
%%%%%%%%%%%%%%%%%%%%%%%%%%%%%% PROOF PROOF PROOF

%%%%%%%%%%%%%%%%%%%%%%%%%%%%%% REMARK REMARK REMARK
\begin{remark} \label{rem:cex2} This remark shows that the counterexample of Theorem \ref{thm:cex} is   sharp. We say that $U\subset \Sigma(n)$ has no odd cycles if it does not contain  tuples of pairs which are images under permutation of $\{1,\ldots, n\}$ of the tuple of pairs \[\{(1,2),(2,3),\ldots,( k-1,k), (1,k )\}\] with $k$ odd. In  the case that $U$ has odd cycles, in each given quadrant of $\R^n$ at least one of the multipliers $K_{\sigma,\ell}^{\eps}$ is trivial for both $\eps=\pm$; we can thus reduce to the case that $U$ has no odd cycles. This case is treated below.
	
Suppose that $\{\upsilon_1,\ldots,\upsilon_s\}\subset \{1,\ldots,n\}$ are such that for all $\sigma \in U$ there exists  $j$ such that $\upsilon_j\in \sigma$: in this case $\{\upsilon_1,\ldots,\upsilon_s\}$ is called \emph{spanning set} of $U$. Notice that for every $U\subset \Sigma$ we may find a spanning set with  $s\leq \lfloor n/2\rfloor$. Arguing in similar fashion as in the proof of Lemma \ref{lemma:annular} we may obtain    the pointwise estimate
\begin{equation}
\label{eq:pw56}
 \big|K_{U,{\cic{\ell}} \,}^{{\cic{\eps}} }(P_{t_1}^{\upsilon_1}\circ \cdots \circ P_{t_s}^{\upsilon_s} f)(x)\big| \lesssim \mathrm M^n _{\mathsf{s}} (P_{t_1}^{\upsilon_1}\circ \cdots \circ P_{t_s}^{\upsilon_s} f) (x),\quad x\in\R^n,
\end{equation}
uniformly over all $t_1,\ldots,t_s \in \R$. Now,  we may use an $s$-parametric version of the Chang-Wilson-Wolff inequality
to reduce estimates for the maximal operator associated to the multipliers $K_{U,\cic{\ell} }^{\cic{\eps}}$ over $\cic\ell \in \mathbb L$ to an $s$-fold Littlewood-Paley square function estimate involving the left hand side of \eqref{eq:pw56} with a loss of $(\log \#\mathbb L)^{\frac s2}$. An application of the  bound \eqref{eq:pw56} as in Proposition \ref{prop:smoothm} 
will thus lead to the estimate
\[ 
\sup_{{\cic{\eps}} \in \{+,-\}^U}  \big\|  \sup_{{\cic{\ell}} \in {\mathbb L} }|   K_{U,{\cic{\ell}}}^{{\cic{\eps}} } f|\big\|_{L^p(\R^n)} \lesssim  (\log \#\mathbb L)^{\frac s2} \|f\|_{L^p(\R^n)},
\]
which, together with the previously made observation that $s$ may be taken  $\leq \lfloor n/2\rfloor$ shows the sharpness of Theorem \ref{thm:cex}; in general the worst case is $U=\{(1,2),(3,4),...,(2\lfloor n/2\rfloor-1,2\lfloor n/2\rfloor)\}$. We leave the details to the interested reader.
\end{remark}
%%%%%%%%%%%%%%%%%%%%%%%%%%%%%% REMARK REMARK REMARK

%%%%%%%%%%%%%%%%%%%%%%%%%%%%%% SECTION SECTION SECTION

%%%%%%%%%%%%%%%%%%%%%%%%%%%%%% SECTION SECTION SECTION

\end{document}